\documentclass[xcolor=dvipsnames,svgnames,table,reqno]{amsart}

\usepackage[colorinlistoftodos]{todonotes}

\input xy
\xyoption{all}
\usepackage{accents}
\usepackage{epsfig}
\usepackage{xcolor}
\usepackage{amsthm}
\usepackage{amssymb}
\usepackage{bbm}
\usepackage{amsmath}
\usepackage{amscd}
\usepackage{amsopn}
\usepackage{graphicx}

\usepackage{tikz}
\usetikzlibrary{shapes,shadows,arrows}
\usepackage{tikz-cd}

\usepackage{xspace}

\usepackage{hhline}
\usepackage{easybmat}
\usepackage{caption}   
\usepackage{relsize}

\usepackage{url}
\usepackage{enumitem, hyperref}\hypersetup{colorlinks}

\usepackage{stmaryrd}

\usepackage[T1]{fontenc}

\colorlet{purpleB70}{blue!70!red}

\colorlet{orangeR65}{red!65!yellow}

\definecolor{red2}{HTML}{d41173}

\definecolor{neongreen}{HTML}{1bf702}

\definecolor{radicalred}{HTML}{FF355E}

\definecolor{denim}{HTML}{1560BD}

\definecolor{darkcyan}{rgb}{0.0, 0.55, 0.55}

\definecolor{cilek}{HTML}{FF43A4}

\definecolor{mor}{HTML}{9F00C5}

\definecolor{phlox}{rgb}{0.87, 0.0, 1.0}

\definecolor{fluorescentpink}{HTML}{FF1493}

\definecolor{napiergreen}{rgb}{0.16, 0.5, 0.0}

\definecolor{kellygreen}{rgb}{0.3, 0.73, 0.09}

\definecolor{parisgreen}{HTML}{ 50C878 }

\definecolor{palatinateblue}{rgb}{0.15, 0.23, 0.89}

\definecolor{ceruleanblue}{rgb}{0.16, 0.32, 0.75}

\definecolor{brandeisblue}{rgb}{0.0, 0.44, 1.0}

\definecolor{KLMblue}{HTML}{0FC0FC}

\definecolor{cinnamon}{rgb}{0.82, 0.41, 0.12}

\definecolor{darkorange}{rgb}{1.0, 0.55, 0.0}

\definecolor{darktangerine}{rgb}{1.0, 0.66, 0.07}

\definecolor{deepcarrotorange}{rgb}{0.91, 0.41, 0.17}

\definecolor{internationalorange}{HTML}{FF4F00}

\definecolor{persimmon}{HTML}{EC5800}

\definecolor{pumpkin}{HTML}{FF7518}

\definecolor{darkred}{rgb}{1,0,0} 
\definecolor{darkgreen}{rgb}{0,0.7,0}
\definecolor{darkblue}{rgb}{0,0,1}

\hypersetup{colorlinks,
linkcolor=darkblue,
filecolor=darkgreen,
urlcolor=darkred,
citecolor=darkgreen}

\makeatletter
\def\reflb#1#2{\begingroup
    #2%
    \def\@currentlabel{#2}%
    \phantomsection\label{#1}\endgroup
}
\makeatother

\numberwithin{equation}{section}
\newtheorem{Theorem}{Theorem}
\numberwithin{Theorem}{section}

\newtheorem{TheoremX}{Theorem}

\newtheorem   {Lemma}[Theorem]{Lemma}

\newtheorem   {Proposition}[Theorem]{Proposition}
\newtheorem   {Corollary}[Theorem]{Corollary}
\theoremstyle {definition}
\newtheorem   {Definition}[Theorem]{Definition}
\theoremstyle {remark}
\newtheorem   {Remark}[Theorem]{Remark}
\newtheorem   {Example}[Theorem]{Example}

\def    \eps    {\epsilon}

\newcommand{\CA}{{\mathcal A}}
\newcommand{\CC}{{\mathcal C}}

\newcommand{\CI}{{\mathcal I}}

\newcommand{\CS}{{\mathcal S}}

\newcommand{\supp}{\operatorname{supp}}

\newcommand{\id}{{\mathit id}}

\newcommand{\const}{{\mathit const}}

\newcommand{\fa}{{\mathfrak a}}

\newcommand{\ty}{\tilde{y}}

\newcommand{\tx}{\tilde{x}}
\newcommand{\hx}{\hat{x}}
\newcommand{\cx}{\check{x}}
\newcommand{\hy}{\hat{y}}
\newcommand{\cy}{\check{y}}
\newcommand{\tz}{\tilde{z}}
\newcommand{\cz}{\check{z}}
\newcommand{\hz}{\hat{z}}

\newcommand{\talpha}{\tilde{\alpha}}

\newcommand{\tH}{\tilde{H}}

\newcommand{\A}{{\mathcal A}}

\newcommand{\Ss}{{\mathcal S}}

\def    \F      {{\mathbb F}}

\def    \R      {{\mathbb R}}

\def    \Z      {{\mathbb Z}}
\def    \N      {{\mathbb N}}
\def    \Q      {{\mathbb Q}}

\def    \CP     {{\mathbb C}{\mathbb P}}

\def    \12     {{\frac{1}{2}}}

\def    \p      {\partial}

\def    \tr     {\operatorname{tr}}

\def    \SH     {\operatorname{SH}}
\def    \Sp     {\operatorname{Sp}}

\def    \HF     {\operatorname{HF}}

\def    \H      {\operatorname{H}}

\def    \CF      {\operatorname{CF}}

\def    \Leb     {\operatorname{Leb}}

\def    \sgn   {\mathit{sgn}}

\def    \hmu   {\operatorname{\hat{\mu}}} 

\def    \MUM   {\operatorname{\mu_{\scriptscriptstyle{Morse}}}}

\def    \pfl   {\p_{\scriptscriptstyle{Fl}}}
\def    \inv   {\mathrm{inv}}

\def    \TSp     {\widetilde{\operatorname{Sp}}}

\newcommand \Cbar {C_{\scriptscriptstyle{bar}}}
\newcommand \nug {\nu_{\scriptscriptstyle{geom}}}
\newcommand \nua {\nu_{\scriptscriptstyle{alg}}}
\newcommand   \slope {\operatorname{\mathit{slope}}}
\newcommand   \rmax {r_{\max}}
\newcommand   \WW {\widehat{W}}

\begin{document}


\setlength{\smallskipamount}{6pt}
\setlength{\medskipamount}{10pt}
\setlength{\bigskipamount}{16pt}





\title [Invariant Sets and Hyperbolic Closed Reeb Orbits]{Invariant
  Sets and Hyperbolic Closed Reeb Orbits}

\author[Erman \c C\. inel\. i]{Erman \c C\. inel\. i}
\author[Viktor Ginzburg]{Viktor L. Ginzburg}
\author[Ba\c sak G\"urel]{Ba\c sak Z. G\"urel}
\author[Marco Mazzucchelli]{Marco Mazzucchelli}

\address{Erman \c C\. inel\. i\newline\indent Department of
  Mathematics, ETH Z\"urich\newline\indent R\"amistrasse 101, 8092
  Z\"urich, Switzerland} \email{erman.cineli@math.ethz.ch}

\address{Viktor Ginzburg\newline\indent Department of Mathematics, UC
  Santa Cruz\newline\indent Santa Cruz, CA 95064, USA}
\email{ginzburg@ucsc.edu}

\address{Ba\c sak G\"urel\newline\indent Department of Mathematics,
  University of Central Florida\newline\indent Orlando, FL 32816, USA}
\email{basak.gurel@ucf.edu}

\address{Marco Mazzucchelli\newline\indent Sorbonne Université,
  Université Paris Cité, CNRS, IMJ-PRG\newline\indent F-75005 Paris,
  France} \email{marco.mazzucchelli@imj-prg.fr}



\date{\today}

\thanks{The work is partially supported by the NSF CAREER award
  DMS-1454342 (BG), NSF grants DMS-2304207 (BG) and DMS-2304206 (VG),
  Simons Foundation grants 855299 (BG), 581382 (VG) and
  MP-TSM-00002529 (VG), ERC Starting Grant 851701 via a postdoctoral
  fellowship (E\c{C}), and the ANR grants CoSyDy, ANR-CE40-0014 and
  COSY, ANR-21-CE40-0002 (MM)}

\begin{abstract}
  We investigate the effect of a hyperbolic (or, more generally,
  isolated as an invariant set) closed Reeb orbit on the dynamics of a
  Reeb flow on the $(2n-1)$-dimensional standard contact sphere,
  extending two results previously known for Hamiltonian
  diffeomorphisms to the Reeb setting. In particular, we show that
  under very mild dynamical convexity type assumptions, the presence
  of one hyperbolic closed orbit implies the existence of infinitely
  many simple closed Reeb orbits. The second main result of the paper
  is a higher-dimensional Reeb analogue of the Le Calvez--Yoccoz
  theorem, asserting that no closed orbit of a non-degenerate
  dynamically convex Reeb pseudo-rotation is locally maximal, i.e.,
  isolated as an invariant set. The key new ingredient of the proofs
  is a Reeb variant of the crossing energy theorem.
\end{abstract}

\maketitle

\tableofcontents

\section{Introduction and main results}
\label{sec:intro+results}

\subsection{Introduction}
\label{sec:intro}
Compact invariant sets that are locally maximal, i.e.,
    the largest invariant sets in some neighborhood, play a
    fundamental role in dynamics. Among the key examples are
    hyperbolic periodic orbits.  In this paper, we investigate the
    impact of hyperbolic -- or, more generally, locally maximal --
    closed Reeb orbits on the dynamics of Reeb flows on the standard
    contact sphere $S^{2n-1}$.  We prove that, under a very mild
dynamical convexity type assumption, the presence of one hyperbolic
closed orbit implies the existence of infinitely many simple closed
Reeb orbits. In a related theorem, we show that for non-degenerate
dynamically convex Reeb flows on the sphere, the same is true when
there is a locally maximal closed Reeb orbit.  These results hold for
all dimensions $2n-1\geq 3$, but they are primarily of interest when
$2n-1\geq 5$: in dimension three the results can be derived from other
known facts in low-dimensional dynamics and three-dimensional contact
topology.

In symplectic dynamics, even a minimal input of localized
hyperbolicity, such as the presence of one or several hyperbolic
periodic orbits, can have a strong impact on non-local dynamics of the
system. Moreover, in some instances hyperbolicity can be replaced by
the weaker condition that the orbit is locally maximal, i.e., isolated
as an invariant set.  This phenomenon manifests itself in a variety of
disparate ways. We illustrate it by the next two examples, only
tangentially related to the main theme of the paper.

The first one is provided by results from, e.g., \cite{Ha, Xi}
asserting that $C^1$-generically the stable and unstable manifolds of
a hyperbolic periodic point of a Hamiltonian diffeomorphism $\varphi$
of a closed symplectic manifold $M$ have transverse non-empty
intersections. As a consequence, $\varphi$ has a horseshoe and
positive topological entropy. (This is a construction somewhat similar
to the $C^1$-closing lemma.) In dimension two, this is also true
$C^\infty$-generically, \cite{LCS1, LCS2}. We also note that for many
manifolds $M$, a Hamiltonian diffeomorphism has infinitely many
hyperbolic periodic points $C^\infty$-generically; see, e.g.,
\cite{CGG:Growth, CGG:Spec} and references therein.

The second example, more of a symplectic geometric nature, is that for
any Hamiltonian diffeomorphism $\varphi\colon M\to M$ with
sufficiently many hyperbolic periodic points the spectral norm
$\gamma(\varphi^k)$ of the iterates is bounded away from zero,
\cite{CGG:Growth}. The lower bound on the required number of
hyperbolic points depends only on the topology of $M$; for instance,
for $S^2$ just one such a point is sufficient.

In this paper we extend to dynamically convex Reeb flows on $S^{2n-1}$
two results about the effect of hyperbolic or locally maximal periodic
points on global dynamics of Hamiltonian diffeomorphisms of $\CP^n$.

The first of these results is that a Hamiltonian diffeomorphism of
$\CP^n$ with a hyperbolic periodic point necessarily has infinitely
many periodic points; see \cite[Thm.\ 1.1]{GG:hyperbolic} and also
\cite{Al1}. (In fact, the theorem holds for a broader class of closed
symplectic manifolds.) In dimension two, i.e., for $S^2=\CP^1$, this
theorem readily follows from the celebrated theorem of Franks from
\cite{Fr1, Fr2} asserting that an area preserving diffeomorphism of
$S^2$ with more than two periodic points must have infinitely many
periodic points; see also \cite{LeC:Fr}. Furthermore, when the stable
and unstable manifolds of the hyperbolic periodic point intersect
transversely and non-trivially, the resulting horseshoe immediately
provides infinitely many periodic points. Moreover, as we have pointed
out above, such intersections exist $C^1$-generically in all
dimensions and $C^\infty$-generically in dimension two. Hence, it is
essential that no intersection condition is imposed in \cite[Thm.\
1.1]{GG:hyperbolic}.

Recently, Franks' theorem has been (partially) generalized to a class
of symplectic manifolds of any dimension including $\CP^n$ under a
minor non-degeneracy requirement; see \cite{Sh:HZ} and also \cite{Al2,
  Al3, CGG:HZ}. This higher-dimensional variant of Franks theorem,
originally conjectured by Hofer and Zehnder in \cite{HZ}, implies
\cite[Thm.\ 1.1]{GG:hyperbolic} in all dimensions.  A different
conjecture inspired by that theorem goes beyond the orbit count and
asserts that a Hamiltonian dynamical system (in a very broad sense)
must have infinitely many simple periodic orbits whenever it has a
periodic orbit which is homologically or geometrically unnecessary,
e.g., a non-contractible or degenerate orbit for Hamiltonian
diffeomorphisms. We refer the reader to, for instance, \cite{Ba1, Ba2,
  Ba3, GG:nc, Gu:nc, Or1, Or2, Su:nc} for some sample results in this
direction.

In the spirit of \cite[Thm.\ 1.1]{GG:hyperbolic}, the first main
result of this paper (Theorem \ref{thm:main1}) is an extension of that
theorem to dynamically convex, in a very loose sense, Reeb flows on
$S^{2n-1}$. Thus this result can be viewed as the first step towards
the contact Franks' theorem in all dimensions.

Our second main result (Theorem \ref{thm:main2}) is a Reeb analogue of
the higher-dimensional Le Calvez--Yoccoz theorem, \cite[Thm.\
4.1]{GG:PR}, on invariant sets of pseudo-rotations. Roughly speaking,
a Hamiltonian pseudo-rotation is a Hamiltonian diffeomorphism with the
minimal possible number of periodic points, where the lower bound is
usually interpreted in terms of Arnold's conjecture. Pseudo-rotations
in dimension two have been extensively studied by dynamical systems
methods (see, e.g., \cite{A-Z, LCY} and references therein) and also
by holomorphic curves techniques, \cite{Br1,Br}. Recently, Floer
theoretic methods have been used to study Hamiltonian pseudo-rotations
in all dimensions; see \cite{CS, GG:PR, JS}. While official
definitions of a pseudo-rotation vary (see, e.g., \cite{CGG:CMD,
  CGG:St, Sh:JMD, Sh:MRL}), for $\CP^n$ they all amount to requiring a
Hamiltonian diffeomorphism to have exactly $n+1$ periodic points which
are then necessarily the fixed points. Pseudo-rotations can have quite
complicated dynamics. For instance, the Anosov--Katok conjugation
method, originally developed in \cite{AK} (see also \cite{FK}) yields
area-preserving diffeomorphisms of $S^2$ with exactly three ergodic
measures: the two fixed points and the area form. The conjugation
method was extended to higher dimensions in the Hamiltonian setting in
\cite{LRS}, leading in particular to a construction of Hamiltonian
diffeomorphisms of $\CP^n$ with exactly $n+2$ ergodic measures: the
fixed points and the volume.

Such pseudo-rotations are uniquely ergodic outside the fixed point
set, and one would expect every orbit to be either periodic or
dense. This is, however, not true. The celebrated Le Calvez--Yoccoz
theorem asserts that for an area-preserving pseudo-rotation of $S^2$
no fixed point is locally maximal, i.e., isolated as an invariant set;
see \cite{LCY} and also \cite{Fr, FM, Sal}.  In \cite[Thm.\
4.1]{GG:PR} this result is generalized to Hamiltonian pseudo-rotations
of $\CP^n$ in all dimensions. We note that here a pseudo-rotation is
not required to be non-degenerate, although this is the case in all
known examples.

The contact analogue of a pseudo-rotation is a Reeb flow with finitely
many periodic orbits. (For the sake of simplicity we are leaving aside
the requirement that the number of closed Reeb orbits is minimal; see
Section \ref{sec:results} and Remark \ref{rmk:number}.) Such flows can
also have very involved dynamics. For instance, ergodic
pseudo-rotations on $S^{2n-1}$ were constructed in \cite{Ka}; these
are Reeb flows on certain $C^\infty$-small perturbations of irrational
ellipsoids. In \cite{AGZ}, pseudo-rotations of $S^3$ with exactly
three invariant measures are constructed by applying the ``contact
suspension'' to Anosov--Katok pseudo-rotations of the disk.  In
Theorem \ref{thm:main2} we show that no closed Reeb orbit of a
dynamically convex, non-degenerate pseudo-rotation of $S^{2n-1}$ is
locally maximal.

In the next section we will precisely state our main results, discuss
them in more detail and also touch upon the key ingredients of the
proofs. Here we only mention that the central new component of the
proofs is a Reeb analogue of the Crossing Energy Theorem from
\cite{CGG:Entropy, GG:hyperbolic, GG:PR}.

\subsection{Main results}
\label{sec:results}
Both of our main results concern the Reeb flow of a contact form
$\alpha$ on $S^{2n-1}$ supporting the standard contact structure. To
state the theorems, recall that a contact form or a Reeb flow is said
to be dynamically convex when $\mu_-(x)\geq n+1$ for all (not necessarily simple, i.e., uniterated) closed
contractible Reeb orbits $x$, where $\mu_-$ is the lower
semi-continuous extension of the Conley--Zehnder index $\mu$. The Reeb
flow on a convex hypersurface in $\R^{2n}$ is dynamically convex,
\cite{HWZ}. Our first result requires a condition similar to dynamical
convexity but notably less restrictive. Namely, denote by $\hmu(x)$
the mean index of a closed Reeb orbit $x$ and by $2\nua(x)$ the
algebraic multiplicity of the eigenvalue 1 of the Poincar\'e return
map of $x$.  We refer the reader to Section \ref{sec:CZ} for a more
detailed discussion of the Conley--Zehnder and mean indices and of
dynamical convexity type conditions and for further references.

\begin{TheoremX}
  \label{thm:main1}
  \label{thm:hyperbolic}
  Assume that $(S^{2n-1},\alpha)$ has a hyperbolic (simple) closed
  Reeb orbit $z$ with $\hmu(z)>0$ and
  $$
  \mu_-(x)\geq \max\big\{3,\, 2+\nua(x)\big\}
  $$
  for all, not necessarily simple, periodic orbits $x$ with
  $\hmu(x)>0$. Then the Reeb flow of $\alpha$ has infinitely many
  simple periodic orbits.
\end{TheoremX}

The condition of the theorem is met when $\alpha$ is dynamically
convex and $2n-1\ge 3$. Indeed, then for all closed Reeb orbits
$\mu_-(x)\geq n+1\geq 3$ and $\mu_-(x)-\nua(x)\geq (n+1)-(n-1)\geq 2$
since $\nua(x)\leq n-1$; see Section \ref{sec:CZ}. Furthermore, note
that $\hmu(x)\geq (n+1)-(n-1)\geq 2$ by \eqref{eq:CZ-mean}. We
emphasize that we do not impose any non-degeneracy requirements on the
Reeb flow in Theorem \ref{thm:main1}.

In contrast, the non-degeneracy and essentially full dynamical
convexity conditions are essential for the proof of our second main
result:

\begin{TheoremX}
  \label{thm:main2}
  Assume that $(S^{2n-1},\alpha)$ with $2n-1\geq 3$ is dynamically
  convex, non-degenerate and its Reeb flow has only finitely many
  simple closed orbits, i.e., the flow is a Reeb pseudo-rotation. Then
  no closed orbit of the flow is locally maximal, i.e., isolated as an
  invariant set.
\end{TheoremX}  

Since hyperbolic orbits are obviously locally maximal, this theorem
would be a stronger statement than Theorem \ref{thm:main1} if not for
the more restrictive conditions on the Reeb flow.

\begin{Remark}
  In fact, as is easy to see from the proof, we prove a slightly
  stronger result than Theorem \ref{thm:main2}. Namely, assume that
  all closed Reeb orbits $x$ of the Reeb flow on
  $(S^{2n-1\geq 3},\alpha)$ with $\hmu(x)>0$ are non-degenerate and
  $\mu(x)\geq n+1$ and that one of such orbits is locally
  maximal. Then the flow has infinitely many simple closed orbits with
  $\hmu>0$.
\end{Remark}

Both of these results are primarily of interest when $2n-1\geq 5$. In
dimension three, Theorem \ref{thm:main1} readily follows from the
existence of a global surface of section, \cite{HWZ}, and Franks'
theorem, \cite{Fr1, Fr2}. Furthermore the non-degenerate version of
Franks' theorem is known to hold in dimension three: every
non-degenerate Reeb flow on a closed contact 3-manifold has either
exactly two closed orbits, which are then elliptic, or infinitely
many; see \cite{CGHP, CDR}. Moreover, for the standard contact sphere
$S^3$ this is true without the non-degeneracy requirement, \cite{CGH,
  GHHM}.

Theorem \ref{thm:main2} holds for any non-degenerate Reeb flow with
finitely many periodic orbits on a closed 3-manifold.  The reason is
that the Le Calvez--Yoccoz theorem is in fact local. To be more
precise, an irrationally elliptic fixed point (or equivalently an
elliptic fixed point which is non-degenerate along with all iterates)
of an area preserving diffeomorphism of a surface is never locally
maximal. This is an immediate consequence of the topological proof of
the Le Calvez--Yoccoz theorem by Franks; see \cite[Prop.\ 3.1]{Fr} and
also \cite{FM}. For the sake of completeness, we have included a proof
in the Appendix (Section \ref{sec:LCY}) -- see Theorem \ref{thm:LCY-F}
and Corollary \ref{cor:LCY-F}, closely following Franks'
argument. (We refer the reader to \cite{CGP, FH, Pr}
    for other relevant symplectic results on the existence of
    invariant sets with certain properties in dimension three.)
However, to the best of our knowledge, nothing like this local result
is known in higher dimensions. In other words, it is not known if a
non-degenerate (with all iterates) elliptic fixed point of a
Hamiltonian diffeomorphism is necessarily locally maximal.

Theorem \ref{thm:main1} and its proof are closely related to the
multiplicity problem for simple closed Reeb orbits on
$S^{2n-1\geq 5}$, and this is where the dynamical convexity--type
condition becomes essential. This problem is an analogue of the Arnold
conjecture for Reeb flows on the standard contact sphere and concerns
with the minimal number of such orbits. Hypothetically, this number is
$n$. The question has been extensively studied and we refer the reader
to, e.g., \cite{CGG:Reeb-HZ, DL2W, GG:LS, GGMa, GK, Lo, LZ} for some
relevant results and further references. However, all these results
require the Reeb flow to meet some additional requirements. Without a
dynamical convexity--type condition (or symmetry), it is not even
known if in general a Reeb flow on $S^{2n-1\geq 5}$ must have more
than one simple closed Reeb orbit, or if there are more than two
simple closed Reeb orbits when the flow is non-degenerate; see
\cite[Rmk.\ 3.3]{Gu:pr} and also \cite{AGKM}.

Likewise, it is tempting to conjecture that a variant of Franks
theorem holds for Reeb flows on $S^{2n-1\geq 5}$: a flow with more
than $n$ simple closed Reeb orbits must have infinitely many such
orbits. Theorem \ref{thm:main1}, the more recent result
  \cite[Thm.\ B]{CGG:Reeb-HZ} and also \cite[Thm.\ 1.7]{Gu:pr} are
the only results known to us supporting this conjecture when
$2n-1\geq 5$.

\subsection{About the proofs}
The proofs of the two main theorems are quite similar and hinge on
three key results. These are the Crossing Energy Theorem (Theorem
\ref{thm:CE}), the Floer Homology Vanishing Theorem (Theorem
\ref{thm:vanishing}) and the Index Recurrence Theorem (Theorem
\ref{thm:IRT}).

In the Hamiltonian setting, the Crossing Energy Theorem asserts that
whenever a 1-periodic orbit $z$ of a Hamiltonian diffeomorphism
$\varphi_H$ is locally maximal (e.g., hyperbolic), every Floer
cylinder $u$ for $\varphi_H^k$ asymptotic to the iterates $z^k$ at
either end has energy $E(u)$ bounded from below by a constant
$c_\infty>0$ independent of $k$; see \cite{CGG:Entropy, GG:hyperbolic,
  GG:PR}. For our purposes independence of $k$ is crucial;
furthermore, for a fixed $k$ the lower bound readily follows from a
suitable variant of Gromov compactness.

A simple proof of the Crossing Energy Theorem in this case is based on
the fact that every loop $t\mapsto u(s,t)$, $t\in S^1_k=\R/k\Z$, is an
$\eps$-pseudo-orbit of the Hamiltonian flow $\varphi_H^t$, i.e., it
deviates from the flow by no more than $\eps$ in time-one, where
$\eps$ is small when $e=E(u)$ is small. In fact, we can take
$\eps=O\big(e^{1/4}\big)$ uniformly in $k$; see \cite[Sec.\ 1.5]{Sa}
or Remark \ref{rmk:salamon}. Then, arguing by contradiction, we assume
that $e\to 0$ for some sequences $k=k_i\to\infty$ and $u=u_{k_i}$. Let
$V$ be a compact isolating neighborhood of $z$. For each $u$, pick
$s\in\R$ such that $u(s,\cdot)$ is tangent to $\p V$ and contained in
$V$. (Strictly speaking, here we have to work with $V\times \R/\Z$
unless $H$ is autonomous.)  Passing to the limit as $k\to\infty$ and
hence $\eps\to 0$, we obtain an integral curve of $\varphi_H^t$
entirely contained in $V$ and different from $z$, which is impossible
since $z$ is locally maximal.

Generalizing this argument to the Reeb and symplectic homology setting
presents several difficulties. First of all, it is not entirely
clear how to state the Crossing Energy Theorem on the level of Reeb
flows and/or symplectic homology. This forces us to work with
admissible Hamiltonians $H$ on the symplectic completion $\WW$ of a
Liouville domain $W$, which are constant on $W$. But then the
1-periodic orbit $\tz$ of $H$ corresponding to a locally maximal
closed orbit $z$ of the Reeb flow on $\p W$ is no longer maximal and,
in addition, a Floer cylinder $u$ for $kH$ asymptotic $\tz^k$ can
hypothetically get arbitrarily close to $W$ where the Hamiltonian
vector field $X_H$ is close to zero. As a consequence, the above
argument breaks down. We show however that, under certain extra
conditions, this does not happen: $u$ remains some distance from $W$;
see Theorem \ref{thm:location}.  This is sufficient to prove a variant
of the Crossing Energy Theorem suitable for our purposes.

We should mention that recently a few other variants of the Crossing
Energy Theorem have been established: for $\CP^n$ by employing
generating functions in \cite{Al1}; for geodesic flows via
finite-dimensional approximations in \cite{GGM}; and finally in
\cite{CGP} for certain holomorphic curves in the symplectization by
using the machinery of feral holomorphic curves developed in
\cite{FH}. Let us, however, emphasize that none of the other variants
of the Crossing Energy Theorem is currently applicable in the setting
of our main theorems, and hence Theorem \ref{thm:CE} is indispensable
for this work. This result is likewise crucial to the more
  recent works \cite{CGGM:Entropy, Fer2}.

The second key ingredient of the proof concerns with the vanishing of
the (non-equivariant) symplectic homology $\SH(W)$. To be more
precise, denote by $\SH^I(\alpha)$ the filtered symplectic homology of
the contact form $\alpha$ on $\p W$.  Then, whenever $\SH(W)=0$, there
exists a constant $C\geq 0$ such that the natural map
$\SH^I(\alpha)\to \SH^{I+C}(\alpha)$ is zero.  In particular, when
$I=\R=I+C$, we have $\SH(\alpha)^I=\SH(W)=\SH^{I+C}(\alpha)$, the map
in question is the identity and we get back the assumption that
$\SH(W)=0$. In other words, the condition that $\SH(W)=0$ implies that
the every bar of the persistence module $\SH^{(-\infty, a)}(\alpha)$
has length at most $C$. (This observation is originally due to Kei
Irie; we refer the reader to \cite{GS} or Section
\ref{sec:def+general} for a proof.)  This is the case, for instance,
when $W$ is a star-shaped domain in $\R^{2n}$ with smooth boundary,
i.e., $\alpha$ is a contact form on the standard contact sphere
$S^{2n-1}$.

In general, the statement is no longer literally true if we replace
$\SH^I(\alpha)$ by the filtered Floer homology $\HF^I(H)$ for an
admissible Hamiltonian $H$ on $\WW$. For instance, $\HF(H)\neq 0$ in
general. However, we show in Theorem \ref{thm:vanishing} that an
analogue of this vanishing result holds for the family of the filtered
Floer homology groups $\HF^I(kH)$, $k\in \N$, with $C$ independent of
$k$, as long as the right end-point of $I$ is within a certain range
which grows linearly with $k$.

The final key ingredient of the proof is the Index Recurrence Theorem
(Theorem \ref{thm:IRT}). This is a symplectic linear algebra or number
theory result roughly asserting that for a finite collection
$\Phi_i\in \TSp(2m)$, the sequences of Conley--Zehnder indices of the
iterates $\phi_i^k$ have a certain recurrence property.  As stated and
used here, the theorem was proved in \cite{GG:LS}, but it can also be
derived from the common index jump theorem from \cite{Lo,
  LZ}. The two theorems are closely related and make a central
component of the proofs of many multiplicity results. For instance,
combined with dynamical convexity, the Index Recurrence Theorem allows
one in the non-degenerate case to construct infinitely many index
intervals of length $2m$ such that each sequence
$\mu\big(\Phi_i^k\big)$ enters each interval at most once.

Theorems \ref{thm:main1} and \ref{thm:main2} are proved by
contradiction. Assuming that the flow has only finitely many simple
periodic orbits we use the Index Recurrence and Crossing Energy
Theorems to find arbitrarily long action intervals $I$ such that for a
suitable Hamiltonian $H$ and all large $k\in \N$, a locally maximal
closed Reeb orbit $z$ gives rise to a non-zero class in $\HF^I(kH)$
with action at the center of the interval. Then the map
$\HF^I(kH)\to \HF^{I+C}(kH)$ is non-zero.  This is impossible by
Theorem \ref{thm:vanishing}. Theorem \ref{thm:main1} requires much
weaker dynamical convexity conditions and no non-degeneracy for other
orbits because $\mu(z^k)=k\mu(z)$ for a hyperbolic orbit $z$. This
enables us to use the Index Recurrence Theorem in a more precise way
tying the index sequences of other closed orbits to $\mu(z^k)$.

\begin{Remark}
  \label{rmk:number}
  In connection with the discussion in Section \ref{sec:results} let
  us point out an interesting discrepancy between Theorem
  \ref{thm:main2} and the higher-dimensional Le Calvez--Yoccoz
  theorem, \cite[Thm.\ 4.1]{GG:PR}, for Hamiltonian pseudo-rotations
  $\varphi$ of $\CP^n$. The former theorem only requires the Reeb flow
  to have finitely many simple closed orbits, while in the latter the
  number of periodic points must be exactly $n$ although the
  Hamiltonian diffeomorphism need not be non-degenerate. From this
  perspective, the conditions of Theorem \ref{thm:main2} are less
  restrictive. This discrepancy is reflected by the difference in the
  definitions of pseudo-rotations of $\CP^n$ and Reeb pseudo-rotations
  of $S^{2n-1}$ which we adopt here. Of course, the higher-dimensional
  Franks theorem from \cite{Sh:HZ} allows us to just require $\varphi$
  to have finitely many periodic orbits along with a minor
  non-degeneracy condition, but this is a substantial extra step using
  a machinery unavailable in the contact setting.
\end{Remark}

The paper is organized as follows. In Section \ref{sec:conv} we set
our conventions and notation. The relevant facts from Floer theory,
mostly quite standard, are assembled in Section \ref{sec:prelim}. In
Section \ref{sec:sympl-hom} we state and discuss in detail the three
key results used in the proofs of the main theorems. We prove the main
results of the paper in Section \ref{sec:main-pf} and the Crossing
Energy Theorem in Section \ref{sec:cross}. Finally, in the Appendix
(Section \ref{sec:LCY}) we recall an argument from \cite{Fr,FM} and
prove the local version of the Le Calvez--Yoccoz theorem.

\medskip\noindent\subsection*{Acknowledgments} The authors are
grateful to Bassam Fayad, Umberto Hry\-ni\-ewicz, Patrice Le Calvez
and Jean-Pierre Marco for useful discussions. Parts of this work were
carried out while the second and third authors were visiting the
IMJ-PRG, Paris, France, in May 2023 and also during the Summer 2023
events \emph{Symplectic Dynamics Workshop} at INdAM, Rome, Italy;
\emph{From Smooth to $C^0$ Symplectic Geometry: Topological Aspects
  and Dynamical Implications Conference} at CIRM, Luminy, France, and
\emph{Workshop on Conservative Dynamics and Symplectic Geometry} at
IMPA, Rio de Janeiro, Brazil. The authors would like to thank these
institutes for their warm hospitality and support.

\section{Conventions and notation}
\label{sec:conv}
In this section we set our conventions and notation, which are mainly
similar to the ones used in \cite{GG:LS}.

\subsection{Hamiltonians and the action functional}
\label{sec:setting}
Even though Theorems \ref{thm:main1} and \ref{thm:main2} concern with
Reeb flows on the sphere $S^{2n-1}$, $2n-1\geq 3$, it is convenient
for the sake of future references to work with more general Liouville
domains than star-shaped domains in $\R^{2n}$. Thus let $\alpha$ be
the contact form on the boundary $M=\p W$ of a Liouville domain
$W^{2n\geq 4}$. For the sake of simplicity we will assume that
$c_1(TW)\mid_{\pi_2}=0$. As usual denote by $\WW$ the symplectic
completion of $W$, i.e.,
$$
\WW=W\cup_M M\times [1,\,\infty)
$$
with the symplectic form $\omega$ extended to $M\times [1,\infty)$ as
$d(r\alpha)$, where $r$ is the coordinate on $[1,\,\infty)$. Sometimes
it is convenient to extend the function $r$ to a collar of $M=\p W$ in
$W$. Thus we can think of $\WW$ as the union of $W$ and
$M\times [1-\eps,\,\infty)$ for small $\eps>0$ with
$M\times [1-\eps,\, 1]$ lying in $W$ and the symplectic form given by
the same formula.

\begin{Example}
  In this paper we are mainly interested in contact forms $\alpha$ on
  the standard contact sphere $M=S^{2n-1}$. In this case, we can take
  a star-shaped domain $W\subset \WW=\R^{2n}$ as the Liouville domain.
\end{Example}

Unless specifically stated otherwise, most of Hamiltonians
$H\colon \WW\to \R$ considered in this paper depend only on $r$
outside $W$, i.e., $H=h(r)$ on $M\times [1,\,\infty)$, where the
function $h\colon [1,\,\infty)\to \R$ is required to meet the
following conditions:
\begin{itemize}
\item $H$ is constant on $W$ and $h$ is monotone increasing;
\item $h$ is convex, i.e., $h''\geq 0$, and $h''>0$ on $(1,\, \rmax)$
  for some $\rmax>1$ depending on~$h$;
\item $h(r)$ is linear, i.e., $h(r)=ar-c$, when $r\geq \rmax$.
\end{itemize}
In other words, the behavior of $h$ changes from a constant on $W$ to
convex in $r$ on $M\times [1,\, \rmax]$, and strictly convex on the
interior, to linear in $r$ on $M\times [\rmax,\, \infty)$.  When
needed, we will denote $\rmax$ by $\rmax(h)$ to indicate its
dependence on $h$.

In what follows, we will refer to $a$ as the \emph{slope} of $H$ (or
$h$) and write $a=\slope(H)$. The slope is always assumed to be
outside the action spectrum of $\alpha$, i.e.,
$a\not\in\CS(\alpha)$. We call $H$ \emph{admissible} when
$H|_W=\const<0$ and \emph{semi-admissible} when $H|_W\equiv 0$. (This
terminology differs somewhat from the standard usage and we emphasize
that \emph{admissible Hamiltonians are not semi-admissible}.) When $H$
satisfies only the last of these conditions, we call it \emph{linear
  at infinity}.

The difference between admissible and semi-admissible Hamiltonians is
just an additive constant: $H- H|_W$ is semi-admissible when $H$ is
admissible. Hence the two Hamiltonians have the same filtered Floer
homology up to an action shift. Our reason for introducing
semi-admissible Hamiltonians is that in the proofs of the main results
we need to work with the Floer homology of a fixed Hamiltonian and its
iterates rather than symplectic homology, and in this case iterated semi-admissible Hamiltonians are more convenient to handle.

The Hamiltonian vector field $X_H$ is determined by the condition
$$
\omega(X_H,\, \cdot)=-dH
$$
and, on $M\times [1,\,\infty)$, 
$$
X_H=h'(r) R_\alpha,
$$
where $R_\alpha$ is the Reeb vector field.  Hence every $T$-periodic
orbit $z$ of the Reeb flow with $T<a$ gives rise to a
1-periodic orbit $\tz=(z,r)$ of $H$, where
\begin{equation}
  \label{eq:level}
h'(r)=T.
\end{equation}
Clearly, $\tz$ lies in the shell $1<r<\rmax$.

The action functional $\CA_H$ is defined as
$$
\CA_H(x)=\int_x\hat{\alpha}-\int_{S^1} H(x(t))\, dt,
$$
where $x\colon S^1=\R/\Z\to \WW$ is a smooth loop in $\WW$ and
$\hat{\alpha}$ is the Liouville primitive $\alpha_W$ of $\omega$ on
$W$ and $\hat{\alpha}=r\alpha$ on $M\times [1,\,\infty)$. (This
  is the negative of the ``standard'' action functional.) More
explicitly, when $x\colon S^1\to M\times [1,\,\infty)$, we
have
$$
\CA_H(x)=\int_{S^1} r(x(t))\alpha\big(x'(t)\big)\, dt
- \int_{S^1} h\big(r(x(t))\big)\, dt.
$$
Thus when $x=\tz=(z,r)$ is a 1-periodic orbit of $H$, the
action can be expressed as a function of $r$ only:
$$
\CA_H(\tz)=A_H(r),
$$
where
\begin{equation}
  \label{eq:AH}
  A_H\colon [1,\,\infty)\to [0,\,\infty)\textrm{ is given by } A_H(r)=r
  h'(r)-h(r).
\end{equation}
Sometimes we will also denote this \emph{action function} by $A_h$.
This is a monotone increasing function, for
$$
A_H'(r)=h'(r)+ r h''(r)-h'(r)=rh''(r)\geq 0.
$$

It is easy to see that
\begin{equation}
  \label{eq:maxAH}
\max A_H=A_H(\rmax)=c\geq a.
\end{equation}
Here the first equality follows from the fact that $A_H$ is monotone
increasing and the second one from that $h$ is linear, i.e.,
$h(r)=ar-c$, on $[\rmax,\,\infty)$.  To prove the inequality, note
first that
$$
h(r_{\max})\leq \int_1^{r_{\max}} h'(r)\,dr\leq a(r_{\max}-1).
$$
Hence, since $h'(r_{\max})=a$, we have
\begin{equation*}
  \begin{split}
c &=a r_{\max}-h(r_{\max})\\
                       &\geq a r_{\max}-a(r_{\max}-1)\\
                       &= a.
  \end{split}
\end{equation*}
Thus, when $r\geq \rmax$, the function $A_H$ is constant, i.e.
\[A_H|_{[\rmax,\infty)}\equiv A_H(\rmax)=c.\] 
For this reason, in what follows we will
usually limit the domain of this function to $[1,\,\rmax]$.

While the function $A_H$ expresses the Hamiltonian action as a
function of $r$ we will also need another variant $\fa_H$ of an action
function, expressing the Hamiltonian action as a function of the
period $T$, i.e., the contact action. In other words, the function
$\fa_H$ translates the contact action to the Hamiltonian action. Thus
$$
\fa_H=A_H\circ (h')^{-1}\colon [0,\,a]\to [0,\,\max A_H=A_H(\rmax)]
$$
is more specifically defined by the condition
\begin{equation}
  \label{eq:fa}
\fa_H(T)=A_H(r),\textrm{ where } h'(r)=T.
\end{equation}
Since $H$ is (semi-)admissible, $h'$ is one-to-one on
$[1,\, \rmax]$, and the inverse $(h')^{-1}$ is defined on
$[0,\, a]$.

Then using the chain rule, we have
$$
\fa'_H(T)=r:=(h')^{-1}(T)\textrm{ and } 1\leq \fa'_H\leq \rmax.
$$
Thus $\fa_H$ a strictly monotone increasing convex $C^1$-function,
which is $C^\infty$ on $(0,\,a )$, with $\fa''_H=\infty$ at $T=0$ and
$T=a$. Furthermore,
\begin{equation}
  \label{eq:fa-ineq}
\fa_{H_1}\leq \fa_{H_0} \textrm{ on } [0,\,\slope(H_0)] \textrm{
    whenever } H_1\geq H_0.
\end{equation}
This inequality is not obvious, and can be proved as follows.
  Let $\ell_i(r)\subset\R^2$ be the tangent line to the graph of $h_i$
  at the point $(r,h_i(r))$. Let $r_i\in[1,\rmax]$ be such that
  $h'_i(r_i)=T$, so that $\fa_{H_i}(T)=A_{H_i}(r_i)$. Since the
  functions $h_i$ are convex, the slope of $\ell_i(r)$ is monotone
  increasing in $r$. The tangent line $\ell_1(r_1)$ passes through the
  point $(r_1,h_1(r_1))$, which is above the graph of $h_0$ (since
  $h_0\leq h_1$). Therefore, since $\ell_0(r_0)$ and $\ell_1(r_1)$
  have the same slope $T$, the tangent line $\ell_1(r_1)$ lies above
  $\ell_0(r_0)$. Because $-A_{H_i}(r)$ is the ordinate of the
    intersection of $\ell_i(r)$ with the vertical axis, we obtain
    $A_{H_0}(r_0)\geq A_{H_1}(r_1)$.

Concluding this section, we note that $\tz$ is well-defined
only as a 1-periodic orbit of the Hamiltonian flow $\varphi_H^t$ of
$H$. This orbit corresponds to the whole circle $\Gamma=\tx(S^1)$ of
fixed points of the time-one map $\varphi_H$ (aka 1-periodic points of
$\varphi_H$). These orbits, however, have the same action, mean index,
etc. In what follows we will ignore this terminological ambiguity.

\subsection{Conley--Zehnder index and dynamical convexity}
\label{sec:CZ}
We refer the reader to, e.g., \cite{SZ} for the definition and a
detailed discussion of the \emph{Conley--Zehnder index} and to, e.g.,
\cite[Sec.\ 4]{GG:LS} or \cite[Sec.\ 2]{GM} or \cite{Lo}. Here we
normalize the Conley--Zehnder index, denoted throughout the paper by
$\mu$, by requiring the flow for $t\in [0,\,1]$ of a small positive
definite quadratic Hamiltonian $Q$ on $\R^{2m}$ to have index
$m$. More generally, when $Q$ is small and non-degenerate, the flow
has index equal to $(\sgn Q)/2$, where $\sgn Q$ is the signature of
$Q$.  In other words, the Conley--Zehnder index of a non-degenerate
critical point of a $C^2$-small autonomous Hamiltonian $H$ on
$\R^{2m}$ is equal to $m-\MUM$, where $\MUM=\MUM(H)$ is the Morse
index of $H$.

We denote by $\mu_\pm\colon \TSp(2m)\to \Z$ the upper and lower
semi-continuous extensions of the Conley--Zehnder index. The mean
index of $\Phi\in \TSp(2m)$ is defined as
$$
\hmu(\Phi)=\lim_{k\to\infty}\frac{\mu_{\pm}\big(\Phi^k\big)}{k}.
$$
This is the unique, up to normalization, homogeneous continuous
quasi-morphism
$$
\hmu\colon \TSp(2m)\to \R;
$$
cf.\ \cite{BG}. It is a standard fact (see, e.g., \cite{SZ}) that
\begin{equation}
  \label{eq:CZ-mean}
\hmu(\Phi)-m\leq \mu_-(\Phi)\leq \mu_+(\Phi)\leq \hmu(\Phi)+m
\end{equation}
and that the first and the last inequalities are strict when $\Phi$ is
non-degenerate.

The assumption that $c_1(TW)\mid_{\pi_2}=0$ guarantees that these
invariants are also defined for a contractible periodic orbit $x$ of
the Reeb flow on $M$ or a Hamiltonian flow on $\WW$, which we denote
by $\mu(x)$, $\mu_\pm(x)$ and $\hmu(x)$ or $\mu(\tx)$, etc. We note
that since the Conley--Zehnder index of a closed Reeb orbit is defined
via a trivialization of the contact structure, dealing with Reeb flows
everywhere above we should set $m=n-1$ where $2n=\dim W$. On the other
hand, for Hamiltonian flows, $m=n$.

A feature of Reeb flows central to our results is that of dynamical
convexity.

\begin{Definition}
  \label{def:DC} 
  The Reeb flow on a $(2n-1)$-dimensional contact manifold is said to
  be \emph{dynamically convex} if for every closed contractible Reeb
  orbit $x$ (or equivalently when $\pi_1$ has no torsion, every simple
  closed contractible Reeb orbit) $\mu_-(x)\geq n+1$.
\end{Definition}

As is shown in \cite{HWZ}, the Reeb flow on a strictly convex
hypersurface in $\R^{2n}$ is dynamically convex. The converse is not
true; see \cite{CE1, CE2}. The geodesic flow of a Finsler metric on
$S^2$ with curvature meeting a certain pinching condition is
dynamically convex, \cite{HP}. While the notion of dynamical convexity
understood literally as in Definition \ref{def:DC} encapsulates an
important class of Reeb flows on the spheres, it is not entirely clear
if it is of serious relevance for other contact manifolds, especially
in higher dimensions. We are not aware of any examples of dynamically
convex geodesic flows on manifolds of dimension $n>2$.  However, some
geodesic flows are close to being dynamically convex. For instance,
all closed geodesics on the standard round sphere $S^{n\geq 3}$ have
index greater than on equal to $n-1$. (For geodesic flows, the Morse
index is equal to the Conley--Zehnder index $\mu_-$.) It is not
difficult to show that the same is true for the geodesic flow of a
Finsler metric which is $C^2$-close to the round metric on $S^n$.

We refer the reader to, e.g., \cite[Sec.\ 4.2]{GG:LS} for a very
detailed treatment of dynamical convexity and also to, e.g., \cite{AM,
  ALM, DL2W, GM, GGMa} for other relevant notions and results.

\section{Preliminaries}
\label{sec:prelim}
In this section we recall basic definitions and results from Floer
theory used in the proofs of Theorems \ref{thm:main1} and
\ref{thm:main2}. None of the results stated here are really new and
most of them are quite standard and can be traced in some form to the
original works, \cite{CFH, Vi}, or found in, e.g., \cite{BO0, BO,
  CO}. When necessary, we give more specific references.

\subsection{Floer equation}
\label{sec:Floer-eq}
Fix an almost complex structure $J$ on $\WW$ satisfying the following
conditions:
\begin{itemize}
\item $J$ is compatible with $\omega$, i.e., $\omega(\cdot,\, J\cdot)$
  is a Riemannian metric, and on the cone $M\times [1,\,\infty)$ we
  have
\item $J r\p /\p r=R_\alpha$ and
\item $J$ preserves $\ker(\alpha)$.  
\end{itemize}
Note that the last two conditions are equivalent to
\begin{equation}
  \label{eq:complex_strc}
dr\circ J=-r\alpha.
\end{equation}
We call such almost complex structures \emph{admissible}. If the
second and the third conditions are satisfied only outside a compact
set and in addition $J$ can be time-dependent and 1-periodic in time
within a compact set, we call $J$ \emph{admissible at infinity}.

Let now $H$ be a Hamiltonian linear at infinity and let $J$ be an
admissible at infinity almost complex structure.  For our purposes it
is convenient to have the $L^2$-anti-gradient of $\CA_H$ adopted as
the Floer equation:
\begin{equation}
  \label{eq:floer_1}
\p_s u =-\nabla_{L^2}\CA_H(u),
\end{equation}
where $u\colon \R\times S^1\to \WW$ with coordinates $s$ on $\R$ and
$t$ on $S^1$. Thus the function $s\mapsto \CA_H\big(u(s,\,\cdot)\big)$
is decreasing.  More explicitly this equation reads
\begin{equation}
\label{eq:floer_2}
\p_s u-J\big(\p_t u- X_H(u)\big)=0.
\end{equation}
Note that in contrast with the standard conventions the leading term
of this equation is not the $\bar{\p}$-operator but the $\p$-operator.
(This is mainly because our action functional is the negative of the standard one.) In other words, when $H\equiv 0$,
solutions of \eqref{eq:floer_2} are anti-holomorphic curves rather
than holomorphic curves. Nonetheless the standard properties of the
solutions of the Floer equation readily translate to our setting,
e.g., via the change of variables $s\mapsto -s$. We will often refer
to solutions of the Floer equation as \emph{Floer cylinders}.

Recall also that the energy of $u$ is by definition
$$
E(u)=\int_{S^1\times\R}\|\p_s u\|^2\,dt\,ds.
$$
When $u$ is asymptotic to $\tx=(x,r^+)$ at $-\infty$ and $\ty=(y,r^-)$
at $+\infty$, we have
\begin{equation}
  \label{eq:Energy-Action}
E(u)=\CA_H(\tx)-\CA_H(\ty)=A_H(r^+)-A_H(r^-),
\end{equation}
where $r^+\geq r^-$ since \eqref{eq:floer_1} is an anti-gradient Floer
equation; see Section \ref{sec:CE}.

We will make extensive use of two properties of Floer cylinders $u$
for admissible or semi-admissible Hamiltonians $H$ and admissible
almost complex structures $J$.

The first one is the standard maximum principle asserting that the
function $r\circ u$ cannot attain a local maximum in the domain mapped
by $u$ into the cone $M\times [1,\,\infty)$ where $r$ is defined. (We
refer the reader to, e.g., \cite{Vi} and also \cite[Sec.\ 2]{FS} for a
direct proof of this fact.) Of course, the same is true for
Hamiltonians linear at infinity and $J$ admissible at infinity in the
domain where the Hamiltonian is a linear function of $r$ and $J$ is
admissible. Moreover, the maximum principle also holds for
continuation Floer trajectories when $h(r)=a(s)r-c(s)$ as long as the
slope $a$ is a non-decreasing function of $s$; there are however no
constraints on the function $c(s)$.  The maximum principle is crucial
to having Floer cylinders and continuation solutions of the Floer
equation contained in a compact region of $\WW$, and thus the Floer
homology and continuation maps for homotopies with non-decreasing
slope defined.

In particular, let $u$ be a solution of the Floer equation asymptotic
to 1-periodic orbits $(x^+,r^+)$ at $-\infty$ and $(x^-,r^-)$ at
$+\infty$. Note that $r^+>r^-$ -- hence the notation -- since with our
conventions $A_H$ is an increasing function and the action is
decreasing along $u$. Then the maximum principle implies that
\begin{equation}
\label{eq:max_prince}
\sup_{\R \times S^1} r\big(u(s,t) \big) \leq r^+
= r\big(u(-\infty,t)\big).
\end{equation}

The second fact we need, which we will refer to as
\emph{Bourgeois--Oancea monotonicity}, is much less standard and goes
back to \cite[p.\ 654]{BO}; see also \cite[Lemma 2.3]{CO}. It asserts
that
\begin{equation}
\label{eq:monotone}
\max_{t \in S^1}  r\big(u(s,t) \big) \geq r^- = r\big(u(\infty,t)\big)
\end{equation}
for all $s\in \R$. For the sake of completeness, we prove
\eqref{eq:monotone} in Section \ref{sec:BO}, closely following the
argument in \cite{CO}.  Note that as a consequence of the maximum
principle and \eqref{eq:monotone}, the left-hand side of
\eqref{eq:monotone} is a monotone decreasing function of $s$ ranging
from $r^+$ at $-\infty$ to $r^-$ at $\infty$.

\subsection{Filtered Floer and symplectic homology}
\label{sec:Floer}

\subsubsection{Floer homology and continuation maps}
\label{sec:Floer-cont}
Let $H$ be a Hamiltonian $H$ linear at infinity such that, as usual,
$\slope(H)\not\in\CS(\alpha)$ and let $I\subset \R$ be an action
interval. Then, regardless of whether $H$ is non-degenerate or not,
the filtered (contractible) Floer homology $\HF^I(H)$ over a fixed
ground field $\F$ is defined as long as the end-points of $I$ are
outside the action spectrum $\CS(H)$ of $H$.  Throughout the paper we
will always assume the latter condition to be met by $I$. For the sake
of brevity, we will write
$$
\HF^\tau(H):=\HF^{(-\infty,\, \tau]}(H).
$$
To define $\HF^I(H)$ it suffices to replace $H$ by a small compactly
supported non-degenerate perturbation.  With our conventions the Floer
homology is graded by the Conley--Zehnder index. Thus a non-degenerate
minimum of $H$ with small Hessian gives rise to a generator of degree
$n$ and a non-degenerate closed Reeb orbit $x$ with Conley--Zehnder
index $m$ gives rise to two generators $\cx$ and $\hx$ with indices
$m$ and $m+1$, respectively. (The Floer complex is discussed in more
detail in Section \ref{sec:Fl-complex}.) We will usually suppress the
grading in the notation.

Let $H_s$, $s\in\R$, be a homotopy between two linear at infinity
Hamiltonians $H_0$ and $H_1$, i.e., $H_s$ is a family of linear at
infinity Hamiltonians such that $H_s=H_0$ when $s$ is close to
$-\infty$ and $H_s=H_1$ when $s$ is close to $+\infty$. (In what
follows we will take the liberty to have homotopies parametrized by
$[0,\,1]$ rather than $\R$.) There are two situations where a homotopy
gives rise to a map in Floer homology.

The first one is when all Hamiltonians $H_s$ have the same slope. Then
the homotopy induces a continuation map
$$
\HF^I(H_0)\to \HF^{I+C}(H_1)
$$
shifting the action filtration by
$$
C=\int_{-\infty}^\infty \max_{z\in\WW}\max\{0,-\p_s H_s (z)\}\,ds.
$$
As a consequence, $\HF(H)=\HF^\tau(H)$ for $\tau>\sup \CS(H)$;
see Lemma \ref{lemma:mon-hom}. Moreover, it is well-known and not hard
to show that $\HF(H_s)$ does not change as long as
$\slope(H_s)$ stays outside $\CS(\alpha)$.

The second one is when $H_s$ is monotone increasing, i.e., the
function $s\mapsto H_s(z)$ is monotone increasing for all $z\in
\WW$. In this case, the function $s\mapsto \slope(H_s)$ is also
monotone increasing. Furthermore, while $\slope(H_0)$ and
$\slope(H_1)$ are still required to be outside $\CS(\alpha)$, the
intermediate slopes $\slope(H_s)$ can pass through the points of
$\CS(\alpha)$. A monotone increasing homotopy induces a map
$$
\HF^I(H_0)\to \HF^{I}(H_1)
$$
preserving the action filtration.

In both cases the fact that the continuation Floer trajectories are
confined to a compact set is a consequence of the maximum principle;
see Section \ref{sec:Floer-eq}.

The Floer homology is insensitive to small perturbations of the
Hamiltonian and the action interval. To be more precise, fix a linear
at infinity Hamiltonian $H$ and an interval $I$. (Here as usual we
require that $\slope(H)\not\in \CS(\alpha)$ and the end-points of $I$
are not in $\CS(H)$.) Assume that the slope of $H'$ is sufficiently
close to the slope of $H$, $H'$ is $C^2$-close to $H$ on the
complement of the domain where they are both linear functions of $r$,
and that the end-points of $I'$ are close to the ones of $I$. Then there is a
natural isomorphism of the Floer homology groups
$$
\HF^I(H)\cong \HF^{I'}(H').
$$

Floer homology carries the pair-of-pants product
$$
\HF_{m_0}^{\tau_0}(H_0)\otimes \HF_{m_1}^{\tau_1}(H_1)\to
\HF_{m_0+m_1-n}^{\tau_0+\tau_1}(H_0+H_1),
$$
where as always we have assumed that the slopes of $H_0$ and $H_1$ and
their sum are not in $\CS(\alpha)$; see, e.g., \cite{AS}.

With our conventions, when $F$ is a semi-admissible Hamiltonian and
$\delta>0$ is small, there is a natural isomorphism
\begin{equation}
  \label{eq:HFlow1}
\HF_*^\delta(F)\cong \H_{n+*}(W; \p W).
\end{equation}
Furthermore, under this isomorphism, the fundamental class
\begin{equation}
  \label{eq:HFlow2}
[W, \p W]\in \H_{2n}(W,\p W)\cong \F\cong \HF_n^\delta(F)
\end{equation}
corresponds to a unit for the pair-of-pants product. To be more
precise, fix a linear at infinity Hamiltonian $H$ and
$\tau\not\in \CS(H)$. For the sake of simplicity, assume in addition
that for a semi-admissible Hamiltonian $F$ as above, $\slope(F)$ and
also $\delta>0$ are sufficiently small. Then the composition of maps
\begin{equation}
  \label{eq:unit}
\HF^\tau(H)=\HF^\tau(H)\otimes
\HF_n^\delta(F)\stackrel{\cong}{\longrightarrow}
\HF^{\tau+\delta}(H+F)\cong \HF^\tau(H)
\end{equation}
with the middle arrow given by the pair-of-pants product is the
identity map.

\subsubsection{Invariance}
\label{sec:invariance}
The results of this section are somewhat less standard although the
methods are. So far, throughout the above discussion, it was
sufficient to have Hamiltonians to only be linear at infinity. (The
Hamiltonian $F$ in \eqref{eq:HFlow1} and \eqref{eq:unit} is an
exception.) In what follows, however, it becomes essential to require
Hamiltonians to be semi-admissible.

Namely, let $H_0\leq H_1$ be two such Hamiltonians.  For the sake of
simplicity, we will assume furthermore that $\rmax(h_1)=\rmax(h_0)$
and use $\rmax$ to denote both of these parameters. (This condition
can be replaced by that $\rmax(h_1)\leq \rmax(h_0)$ with suitable
wording modifications.)

Let
$$
f=\fa_{H_1}\circ \fa_{H_0}^{-1}\colon [0,\, A_{H_0}(\rmax)]\to [0,\,
A_{H_1}(\rmax)].
$$
The function $f$ is monotone as a composition of two monotone
increasing functions. Furthermore, as is easy to see,
\begin{equation}
  \label{eq:f-sectra}
  f\big(\CS(H_0)\big)=f\big([0,\, A_{H_0}(\rmax)]\big)\cap\CS(H_1),
\end{equation}
i.e., $f$ gives rise to a one-to-one correspondence between the action
spectra as long as the target is in the range of $f$.

\begin{Proposition}
  \label{prop:f}
  We have $f(\tau)\leq \tau$ for all $\tau$.  Furthermore, there are
  isomorphisms of Floer homology groups
  \begin{equation}
    \label{eq:isom-f}
    \HF^\tau(H_0)\stackrel{\cong}{\longrightarrow}
    \HF^{f(\tau)}(H_1),\qquad \forall
    \tau\in[0, A_{H_0}(\rmax))\setminus\CS(H_0).
\end{equation}
These isomorphisms are natural in the sense that they commute with
maps induced by inclusion of action intervals, monotone homotopies,
etc. In particular, whenever $\beta$ is in the range of $f$,
the composition
$$
\HF^\beta(H_0)\to \HF^{\beta}(H_1)\to
\HF^{f^{-1}(\beta)}(H_0),
$$
where the first arrow comes from a monotone increasing homotopy and
the second is the inverse of \eqref{eq:isom-f}, is induced by the
inclusion of action intervals: $\beta\leq f^{-1}(\beta)$.
\end{Proposition}

Of course, a similar result holds for admissible Hamiltonians, but
then we also have to take into account the shift of actions. For the
sake of completeness we include a proof of the proposition.

\begin{proof} The fact that $f(\tau)\leq \tau$ follows from that
  $\fa_{H_1}\leq \fa_{H_0}$ when $H_1\geq H_0$; see
  \eqref{eq:fa-ineq}. We construct the isomorphism \eqref{eq:isom-f}
  in two steps.

  Let $r_0\in [1,\, \rmax)$ be uniquely determined by the condition
  that $A_{H_0}(r_0)=\tau$. Pick an intermediate semi-admissible
  Hamiltonian $H_{01}=h_{01}(r)$ with
  $\rmax(h_{01})=\rmax(h_0)=\rmax(h_1)$ and the following properties
  \begin{itemize}
  \item $h_0\leq h_{01}\leq h_1$;
  \item $h_{01}=h_0$ on $[1,\, r_0]$;
  \item $\slope(h_{01})=\slope(h_1)$.
  \end{itemize}
  We note that $\tau\not\in \CS(H_{01})$.
  
  Then the monotone increasing linear homotopy from $H_0$ to $H_{01}$
  induces an isomorphism
\begin{equation}
  \label{eq:isom-f1}
\HF^\tau(H_0)\stackrel{\cong}{\longrightarrow} \HF^\tau(H_{01}).
\end{equation}
This is a consequence of the fact that by the maximum principle (see
Section \ref{sec:Floer-eq}) all Floer cylinders for $H_0$ and $H_{01}$
and Floer continuation trajectories starting at 1-periodic orbits with
action less than or equal $\tau$ lie on the region $r\leq r_0$ or, to
be more precise, $W\cup \big(M\times [1,\, r_0]\big)$; cf.\ Section
\ref{sec:Fl-complex}.

The second step is based on the following standard lemma which
essentially goes back to \cite{CFH,Vi}.

\begin{Lemma}
  \label{lemma:mon-hom}
  Let $F_s$ be a homotopy of Hamiltonians linear at infinity such that
  all Hamiltonians $F_s$ have the same slope, and let $I_s$ be a
  family of intervals continuously depending on $s$ such that for all
  $s$ the end-points of $I_s$ are outside $\CS(F_s)$. Then there is a
  natural isomorphism
  \begin{equation}
    \label{eq:isom-mon-hom}
    \HF^{I_0}(F_0)\stackrel{\cong}{\longrightarrow} \HF^{I_1}(F_1).
  \end{equation}
\end{Lemma}
The lemma is proved by breaking down the homotopy into a concatenation
of homotopies such that for each of them either the interval or the
Hamiltonian is independent of $s$. For a fixed Hamiltonian and varying
interval, the assertion follows directly from the definition. For a
fixed interval and varying Hamiltonian, it is a consequence of the
fact that the ``inverse'' homotopy induces the inverse map. This is
the point where the condition that $\slope(F_s)=\const$, and hence a
homotopy need not necessarily be monotone increasing, enters the
picture.

In the setting of the proposition, let $F_s$ be a linear monotone
increasing homotopy from $F_0=H_{01}$ to $F_1=H_1$. Denote by
$f_s=\fa_{F_s}\circ \fa_{F_0}^{-1}$ the resulting family of maps and
note that $f_s(\tau)\not\in \CS(F_s)$ by \eqref{eq:f-sectra} since
$\tau\not\in\CS(F_0)$. Applying the lemma to the family of intervals
$I_s=(-\infty, \, f_s(\tau)]$, we obtain an isomorphism
\begin{equation}
  \label{eq:isom-f2}
  \HF^\tau(H_{01})\stackrel{\cong}{\longrightarrow}
  \HF^{f_1(\tau)}(H_{1})= \HF^{f(\tau)}(H_{1}).
\end{equation}
In the last identity we have used the fact that $f=f_1$ on
$[0,\,\tau]$, i.e., in self-explanatory notation
$f_{H_0H_1}=f_{H_{01}H_1}$ on this interval, since $h_{01}=h_0$ on
$[1,\, r_0]$.

The desired isomorphism \eqref{eq:isom-f} is now defined as the
composition of the isomorphisms \eqref{eq:isom-f1} and
\eqref{eq:isom-f2}.  The last assertion of the proposition follows
immediately from this construction.
\end{proof}

\begin{Remark}
  The condition that $\slope(F_s)=\const$ in Lemma \ref{lemma:mon-hom}
  is essential. For instance, let $F_s$ be a monotone increasing
  homotopy from $F_0$ to $F_1$. Let $I_s=(-\infty,\, \tau]$ where
  $\tau> \sup_{s\in \R} \max \CS( F_s)$. Then
  $\HF^I(F_0)=\HF(F_0)$ and $\HF^I(F_1)=\HF(F_1)$ are not in general
  isomorphic when $\slope(F_s)$ crosses $\CS(\alpha)$ as one can see
  already from the example of the round sphere (i.e., the unit
    sphere in $\R^{2n}$ equipped with the restriction of the standard
    Liouville form). However, one can somewhat relax this condition
  by requiring that $\slope(F_s)$ stays outside $\CS(\alpha)$.
\end{Remark}

\subsubsection{Local Floer homology}
\label{sec:localFl}
When $x$ is an isolated $T$-periodic orbit of $\alpha$ with
$T<\slope(H)$ the 1-periodic orbit $\tx$ of $H$ is also isolated as a
1-periodic orbit of the flow of $H$. As a consequence, we obtain an
isolated circle $\Gamma=\tx(S^1)$ of fixed points of $\varphi_H$. We
denote by $\HF(\tx)$ the \emph{local Floer homology} of $\Gamma$; see
\cite{Fl1,Fl2} and also \cite{Fen, GG10, McL}.  By definition, the
support $\supp\HF(\tx)$ is the range of degrees for which this
homology is non-trivial. Clearly, by \eqref{eq:CZ-mean} with $m=n-1$,
\begin{equation}
  \label{eq:support}
  \supp\HF(\tx)\subset [\mu_-(x),\,\mu_+(x)+1]
  \subset [\hmu(x)-n+1,\,\hmu(x)+n].
\end{equation}
For instance, when $x$ is non-degenerate the support of $\HF(\tx)$
comprises exactly two points: $\mu(x)$ and $\mu(x)+1$. More generally,
$\HF_*(\tx)=\HF_*(\psi)\oplus \HF_{*-1}(\psi)$, where $\psi$ is the
germ of the Poincar\'e return map of $x$; see \cite{Fen}. However, we
will not use this fact.

\subsubsection{Symplectic homology}
\label{sec:SH}
Let $I$ be an interval, and as always assume that the end-points of
$I$ are not in $\CS(\alpha)$. The \emph{symplectic homology}
$\SH^I(\alpha)$ is defined as
\begin{equation}
  \label{eq:SH}
\SH^I(\alpha):=\varinjlim_H \HF^I(H),
\end{equation}
where the limit is over all Hamiltonians linear at infinity and such
that $H|_W<0$. Since admissible (but not semi-admissible) Hamiltonians
form a co-final family in the class of linear at
    infinity Hamiltonians on $\WW$, we can limit $H$ to this
class. When working with this definition, it is useful to keep in mind
that
$$
\CS(H)\to\{ 0\}\cup\CS(\alpha)
$$
uniformly on compactly intervals. Clearly, $\SH^I(\alpha)$ is a
$\Z$-graded vector space over $\F$ with the grading coming from the
Conley--Zehnder index. For the sake of brevity we will write
$\SH^\tau(\alpha)$ when $I=(-\infty,\, \tau]$.

\begin{Remark}
  \label{rmk:adm}
  In \eqref{eq:SH}, we could have required that $H|_W\leq 0$ rather
  than that $H|_W<0$, or equivalently required $H$ to be
  semi-admissible or admissible. This would result in the same groups
  $\SH^I(\alpha)$, but also keep us from having convenient choices of
  co-final sequences. For instance, let $H$ be a semi-admissible
  Hamiltonian. Pick two sequences of positive numbers:
  $\lambda_i\to\infty$ and $\eps_i\to 0$. Then the sequence
  $H_i=\lambda_i H-\eps_i$ is co-final in the class of admissible
  Hamiltonians. However, neither this sequence nor the sequence
  $\lambda_i H$ is co-final when semi-admissible Hamiltonians is added
  to the class.
\end{Remark}

From a somewhat different perspective, $\SH^\tau(\alpha)$ can also be
defined as follows; cf.\ \cite{Vi}. For $\tau\not\in\CS(\alpha)$,
consider the function $F_\tau\colon \WW\to [0,\,\infty)$ given by
$$
F_\tau(z)=
\begin{cases}
  0 & z\in W,\\
  \tau\big(r(z)-1\big) & z\in M\times [1,\, \infty).
\end{cases}
$$
While this function is only continuous, its Floer homology is
obviously defined (by continuity). For instance, we can set
$$
\HF^I(F_\tau):=\varinjlim_{H\leq F_\tau}\HF^I(H)
$$
where the limit is taken over admissible or semi-admissible
Hamiltonians bounded from above by $F_\tau$. It is not hard to see
that
$$
\CS(H)\to \{ 0\}\cup\big(\CS(\alpha)\cap[0,\, \tau]\big),
$$
and hence it suffices to require that the end-points of $I$ are not in
$\{ 0\}\cup\big(\CS(\alpha)\cap[0,\, \tau]\big)$. 

\begin{Corollary}
  For $I\subset \R$, we have
  $$
  \HF^I(F_\tau)=\SH^{(-\infty,\,\tau]\cap I}(\alpha).
  $$
  In particular, $\HF(F_\tau)=\SH^{\tau}(\alpha).$
\end{Corollary}

This corollary is essentially a consequence of Proposition
\ref{prop:f} with some wording modifications; its proof is routine and
for the sake of brevity we omit it.

A well-known, but crucial for our purposes, fact is that $\SH(\alpha)$
vanishes whenever $W$ is displaceable in $\WW$. For instance,
$\SH(\alpha)=0$ when $W$ is a ball in $\WW=\R^{2n}$. In the latter
case, vanishing of $\SH(\alpha)$ is established in \cite{Vi}. The
general case is proved in \cite{CFO} via Rabinowitz--Floer homology
and direct proofs are given in \cite{Su} and~\cite{GS}.

On the level of specific Hamiltonians this fact is reflected by the
following lemma which follows from the definition of symplectic
homology as a direct limit and Remark \ref{rmk:adm}.

\begin{Lemma}
  \label{lem:vanishing}
  Assume that $\SH(\alpha)=0$. Then for any semi-admissible
  Hamiltonian $H$ and any $I\subset \R$ there exist constants
  $\lambda\geq 1$ and $C\geq 0$ such that the natural
  inclusion/homotopy map
$$
\HF^I(H)\to \HF^{I+C}(\lambda H)
$$
is zero.
\end{Lemma}

\subsection{Floer complexes and graphs}
\label{sec:Fl-complex}
The proofs of the main theorems require working on the level of Floer
complexes rather than Floer homology. The construction of the Floer
complex is quite standard and in this section we just briefly spell
out the necessary definitions.

\subsubsection{Non-degenerate case.}
\label{sec:nondeg}
Assume first that the Reeb flow of $\alpha$ is non-degenerate. Then a
(semi-)admissible Hamiltonian $H$ is Morse--Bott non-degenerate except
for the critical set $W$. To deal with this minor issue, we fix an
autonomous $C^2$-small perturbation $\tH$ of $H$ such that $\tH-H$ is
supported in a small neighborhood of $W$ and $\tH$ is Morse. Now $\tH$
is Morse--Bott non-degenerate and all constant 1-periodic orbits of
$\tH$ are non-degenerate.

Throughout the paper, we will consider Floer homology
  groups with coefficients in $\Z_2=\{0,1\}$.  We will define the
\emph{Floer complex} $\CF(H)$ by applying the standard Morse--Bott
complex construction using ``cascades'' to the non-constant 1-periodic
orbits of $H$; see, e.g., \cite{BH10, BH13, Bo, BO0, Fu, HN} and
references therein. (The setting of \cite{BO0} is the closest to ours.)  Let $x$ be a closed Reeb orbit of $\alpha$
with period $T<a$ and let $\tx(x,r)$ be the corresponding orbit of
$H$; see Section \ref{sec:setting}. All non-constant 1-periodic orbits
of $H$ have this form. Here we think of $\tx$ as a, not necessarily
simple, 1-periodic orbit of the flow of $H$. It gives rise to a family
of fixed points of $\varphi_H$ parametrized by the circle
$\Gamma_x=\tx(S^1)\subset \WW$ which we identify with
$x(S^1)\subset M$. Fix a Morse function $f_x$ on $\Gamma$ with exactly
one maximum and one minumum and a Riemannian metric which we suppress
in the notation. (Note that different orbits $x$ can give rise to the
same set $\Gamma$ and in this case we are allowed to take different
functions $f_x$, although this is not necessary for our purposes.)

Each orbit $\tx$ gives rise to two generators $\hx$ and $\cx$ of
$\CF(H)$ with grading $|\hx|=\mu(x)+1$ and $|\cx|=\mu(x)$,
corresponding to the maximum and the minimum of $f_x$. Each constant
orbit of $\tH$ gives rise to one generator. (In the proofs we are only
interested in the part of the complex where the actions and indices
are large. Clearly, the generators coming from $W$ do not contribute
to this part.)

The Floer differential $\pfl$ counts cascades: concatenations of
integral curves of $-\nabla f_x$ and Floer trajectories. For instance,
when $|\hx|=|\cy |+1$, a cascade from $\hx$ to $\cy$ comprises an
integral curve on $-\nabla f_x$ from $\hx$ to some point
$z\in \Gamma_x$, a solution of the Floer equation asymptotic at
$-\infty$ to a parametrization of $\tx$ with $\tx(0)=z$ and to some
parametrization of $\ty$ at $+\infty$ and finally an integral curve of
$-\nabla f_y$ from $\ty(0)$ to $\cy$.  The coefficient with which
$\cy$ enters $\pfl \hx$ is an algebraic number of such
(unparametrized) cascades. Cascades from $\hx$ to $\hy$ or from $\cx$
to $\cy$ or $\hy$ are defined in a similar fashion, as long as the
degree difference is 1, but now at least one of the integral curves is
constant. Note also that $\cx$ enters $\pfl \hx$ with zero
coefficients since the Morse function $f_x$ is perfect.

While the generators of the complex are essentially determined by $H$
(and its perturbation $\tH$), the differential depends in addition on
some auxiliary data: the functions $f_x$ and metrics on $\Gamma_x$ and
an almost complex structure $J$ admissible at infinity. To ensure
regularity, it suffices to take $J$ generic in this class, \cite{FHS}.
In particular, we can take $J$ arbitrarily $C^\infty$-close to a fixed
admissible almost complex structure. The Floer complex $\CF(H)$ is
still filtered by the action functional $\CA_H$. Hence, while the
generators $\hx$ and $\cx$ lie on the same action level, the
differential $\pfl$ is strictly action decreasing.

We find it convenient to think in terms of the \emph{Floer graph} of
$H$; see \cite{CGG:HZ}. The vertices of the graph are the generators
of the complex and two generators are connected by an arrow whenever
one of them enters the Floer differential of the other with non-zero
coefficient. The \emph{length of the arrow} is by definition the
action difference. Keeping in mind that the generators outside $W$
come in pairs, we will also say that $\tx$ (or $x$) is connected by an
arrow to $\ty$ (or to $y$) when $\hx$ or $\cx$ is connected to $\hy$
or $\cy$. The Floer graph carries the same information as the
Morse--Bott Floer complex of $H$. In particular, it also depends on
all the auxiliary structures and choices used in the construction of
the Floer Morse--Bott differential.

\subsubsection{Degenerate case}
\label{sec:degen}
Assume now that the Reeb flow of $\alpha$ is degenerate and let $H$ be
a (semi-)admissible Hamiltonian for $\alpha$. As always we require
that $\slope(H)\not\in\CS(\alpha)$. Then the Floer complex of $H$ is
defined by replacing it with a $C^\infty$-small perturbation
$\tH$. Here we specify two ways to do this.

In the first one, we simply take as $\tH$ a $C^\infty$-small
autonomous maximally non-degenerate perturbation of $H$ which has the
same or nearly the same slope as $H$. Then $\tH$ is Morse--Bott and we
can apply the construction from the previous section.

The second approach is based on perturbing the contact form rather
than directly $H$. To this end, fix a $C^\infty$-small non-degenerate
perturbation $\talpha= g\alpha$ of $\alpha$, where $g\colon M\to \R$
is $C^\infty$-close to 1.  Thus $\talpha$ is the contact form on the
hypersurface $\{r=g\}$ in $\WW$. (We may assume that $g\geq 1$ or
extend $r$ to a collar of $M$ in $W$.) As a result, while the
Liouville domain $W$ is slightly affected, we can keep $\WW$
unchanged. The $r$-coordinate is however also affected and $H$ is no
longer admissible for $\talpha$. We replace it with the new
Hamiltonian $\tH$ obtained by the change of variables $r\mapsto r/g$
on the cylindrical part of $\WW$ and also perturb it in a neighborhood of $W$ as in the non-degenerate case.  Now, the resulting Hamiltonian $\tH$ is Morse--Bott and admissible for $\talpha$. We fix the necessary auxiliary structures and apply the
construction from Section \ref{sec:nondeg} to $\tH$.  By definition,
the Floer complex of $H$ is the Floer complex of~$\tH$.

In both cases, the Hamiltonian $\tH$ is $C^\infty$-close to $H$ on
compact sets and in the first construction we can even have $\tH=H$ at
infinity.

Assume next that the closed Reeb orbits of $\alpha$ with period less
than $\slope(H)$ are isolated. Then so are the non-constant 1-periodic
orbits of $H$. As a consequence, all non-constant 1-periodic orbits of
$\tH$ arise from non-constant 1-periodic orbits $\tx$ of $H$ splitting
into Morse--Bott non-degenerate orbits. In other words, non-constant
1-periodic orbits of $\tH$ and the corresponding generators of
$\CF(\tH)$ come in clusters labeled by the orbits of $H$. Limiting the
Floer differential to a cluster we obtain a complex, and the homology
of this complex is the local Floer homology $\HF(\tx)$; see Section
\ref{sec:localFl}.

The \emph{reduced Floer graph} of $H$ is then defined as follows; cf.\
\cite[Sec.\ 5.2]{CGG:HZ}. The vertices are the 1-periodic orbits of
$H$ or equivalently the underlying closed Reeb orbits of
$\alpha$. Each vertex is also labeled by $\HF(\tx)$. In addition, the
graph has an extra vertex corresponding to $W$, which is labeled by
$$
\HF^\delta(H)=\HF^\delta(\tH)=\H(W;\p W)[-n]
$$
for a sufficiently small $\delta>0$; see \eqref{eq:HFlow1} and
\eqref{eq:HFlow2}.

The arrows of the reduced Floer graph are defined by using the action
filtration spectral sequence and then collapsing it into one complex
as in \cite[Sec.\ 2.1.3 and 2.5]{GG:LS}. Namely, the $E_2$-term of
this spectral sequence is
\begin{equation}
  \label{eq:E2}
\CC = \bigoplus_{x}\HF(\tx)\oplus \HF^\delta(H)
\end{equation}
Then \cite[Lemma 2.8]{GG:LS} gives a way to organize the higher level
differentials in the spectral sequence into one differential
$\p\colon E_2\to E_2$ so that the resulting homology is isomorphic to
$\HF(H)$ or, to be more precise, to the graded space associated with
the filtration of $\HF(H)$ by the images of the maps
$\HF^\tau(H)\to \HF(H)$. Denote by
$$
\p_{\tx\ty}\colon \HF(\tx)\to \HF(\ty)
$$
the $\HF(\ty)$-component of the restriction of $\p$ to $\HF(\tx)$.
The vertices $x$ and $y$ or equivalently $\tx$ and $\ty$ are connected
by an arrow if $\p_{\tx\ty}\neq 0$ and then the arrow is labeled by
this map. Thus the graph carries exactly the same information as the
complex $(\CC,\p)$ together with the decomposition \eqref{eq:E2}. The
length of an arrow is again the action difference. The arrows to or
from the vertex $W$ are defined in a similar fashion.

\begin{Example}
  \label{ex:no-arrows}
  Clearly, the range of degrees of the generators in the cluster
  corresponding to $x$ is contained in
$$
[\mu_-(x),\,\mu_+(x)+1]\subset [\hmu(x)-n+1,\,\hmu(x)+n].
$$
Hence, $x$ and $y$ are never connected by an arrow when
$\hmu(x)-\hmu(y)>2n$.

Likewise, assume that every Floer cylinder $u$ asymptotic to $\tx$ has
energy $E(u)>\sigma$; see Section \ref{sec:CE} for the precise
definition. Then all arrows to or from $x$ have length greater
than~$\sigma$. We will repeatedly use both of these facts in the
proofs of the main theorems.
\end{Example}  

\begin{Remark}
  Even when $\alpha$ is non-degenerate, the reduced Floer graph
  differs from the Floer graph in two ways. First, in the reduced
  Floer graph the vertices $\cx$ and $\hx$ are lumped together. (The
  same is true for the generators coming from the critical points of
  $\tH$ in $W$.) But equally importantly the reduced Floer graph may
  have fewer arrows. This can be the case already when the
  construction is applied to a Morse function, and hence the two
  graphs have exactly the same vertex set; cf.\ \cite[Rmk.\
  2.10]{GG:LS}. Furthermore, it is not clear to us to what extent the
  reduced Floer graph depends on the auxiliary data: the perturbation
  $\tH$, the Morse--Bott data, the almost complex structure, etc.
\end{Remark}

\section{Background results}
\label{sec:sympl-hom}
In this section we discuss three results central to the proofs of the
main theorems. These are the Crossing Energy Theorem for admissible
Hamiltonians (Theorem \ref{thm:CE}); Theorem \ref{thm:vanishing}
stating, generally speaking, that the barcode of a (semi-)admissible
Hamiltonian is a priori bounded; and finally the Index Recurrence
Theorem (Theorem \ref{thm:IRT}).

\subsection{Crossing energy theorem}
\label{sec:CE}
The first key ingredient of the proofs is the Crossing Energy Theorem
(Theorem \ref{thm:CE}). To state this result, we start by recalling
some terminology.

Let $z$ be a closed Reeb orbit of $\alpha$ with period $T$. We say
that $z$ is \emph{isolated} (as a periodic orbit) if for every $T'> T$
it is isolated among periodic orbits with period less than
$T'$. Clearly, all periodic orbits of $\alpha$ are isolated if and
only if for every $T'$ the number of periodic orbits with period less
than $T'$ is finite. A stronger requirement is that $z$ is
\emph{isolated as an invariant set} or \emph{locally maximal}, i.e.,
$z$ has an open neighborhood $V$ which contains no invariant sets other than the image $z(\R/T\Z)$.  We call $V$ an \emph{isolating
  neighborhood}. By shrinking $V$ if necessary, one can always a find a compact neighborhood with the same property, which we will refer to as a compact isolating neighborhood. These definitions extend verbatim to any flow.

For instance, a non-degenerate periodic orbit is isolated as a
periodic orbit but not necessarily as an invariant set. A hyperbolic
periodic orbit is isolated as an invariant set.

Let $H(x,r)=h(r)$ be a (semi-)admissible Hamiltonian with 
$a=\slope(H)>T$, and let $\tz=(z,r_*)$ be the corresponding 1-periodic orbit of the Hamiltonian flow of $H$, where 
$$
h'(r_*)=T
$$
by \eqref{eq:level}. In particular, $r_*< \rmax$; see Section
\ref{sec:setting}. Note that $\tz$ is isolated as a 1-periodic orbit
of $H$ if $z$ is isolated. Moreover, $\tz$ is Morse--Bott non-degenerate when $z$ is non-degenerate. However, $\tz$ is never isolated as an invariant
set, as it belongs to a cylinder foliated by periodic orbits. Then
$$
\tz^k=(z^k,r_*)
$$
is also a 1-periodic orbit of $kH$ or equivalently a $k$-periodic
orbit of $H$.

Let $u\colon \R \times S^1 \to\WW$ be a Floer cylinder of $H$. We say
that $u$ is \emph{asymptotic} to $\tz$ at $-\infty$ if there exists a
sequence $s_i\to -\infty$ such that $u(s_i,\cdot)\to \tz$ in the
$C^1$-sense, up to the choice of the initial condition on $\tz$ which
might depend on $s_i$. (To be more precise, for some
  sequences $s_i\to -\infty$ and $\theta_i\in S^1$, the maps
  $t\mapsto u(s_i,t+\theta_i)$ $C^1$ converge to the map $t\mapsto \tz(t)$.)
This is equivalent to $u(s, \cdot)\to \tz$ in the $C^\infty$-sense as
$s\to -\infty$ when $z$ is non-degenerate, \cite{Bo}. (Likewise, $u$
is asymptotic to $\tz$ at $+\infty$ when $s_i\to +\infty$, etc.)

In general, $u$ can be asymptotic to more than one orbit $\tz$ at the
same end when $z$ is not isolated. However,
$\CA_H(\tz)=\lim\CA_H\big(u(s_i,\cdot)\big)$, and hence $\CA_H(\tz)$
is independent of the choice of $\tz$. Furthermore,
\eqref{eq:Energy-Action} holds: $E(u)$ is the difference of actions of
the orbits which $u$ is asymptotic to at $\pm\infty$. It is a standard
fact that $u$ is asymptotic to some 1-periodic orbits of $H$ at
$\pm\infty$ if and only if $E(u)<\infty$; see \cite[Sec.\ 1.5]{Sa}.

Next, assume that $z$ is isolated as a closed Reeb orbit or,
equivalently, $\tz$ is isolated as a 1-periodic orbit of the flow of
$H$. Then, as is easy to see, $\tz$ is unique and $u(s,\cdot)\to \tz$
as $s\to-\infty$ in the $C^1$-sense, up to the choice of an initial
condition on $\tz$ which might depend on $s$. This is a consequence of
the fact that
$$
E\big(u|_{(-\infty,s_i]\times S^1}\big)\to 0\textrm{ as } s_i\to -\infty
$$
since $\CA_H\big(u(s,\cdot)\big)$ is a monotone function of $s$ and of
the argument in \cite[Sec.\ 1.5]{Sa}; see also Remark
\ref{rmk:salamon}.  Therefore, for every tubular neighborhood
$U\subset \WW$ of $\tz(S^1)$, there exists $s_0\in \R$ such that
$u\big((-\infty,\, s_0]\times S^1\big)\subset U$, and for $s\leq s_0$
the loop $u(s,\cdot)$ is homotopic to $\tz$ in $U$. For instance, when
$\tz$ is $k$-iterated, the free homotopy class of $u(s,\cdot)$ is the
$k$th multiple of a generator of $\pi_1(U) \cong \Z$.

\begin{Theorem}[Crossing Energy Theorem]
\label{thm:CE}
Assume that $z$ is a locally maximal $T$-periodic Reeb orbit
of $\alpha$.  Let $H(r,x) =h(r)$ be a semi-admissible or admissible
Hamiltonian with $\slope(H)>T$, meeting the additional requirement
that
\begin{equation}
  \label{eq:h'''1}
  h'''\geq 0 \textrm{ on } [1,\, r_*+\delta]
\end{equation}
for some $\delta>0$, where $h'(r_*)= T$. Fix an admissible
almost complex structure $J$. Then there exists $\sigma >0$ such that
$E(u) \geq \sigma$ for any $k \in \N$ and any Floer cylinder
$u \colon \R \times S^1 \to \WW$ of $kH$ asymptotic, at either end, to~$\tz^k = (z^k, r_*)$.
\end{Theorem}

A few remarks are due regarding the statement of this theorem, which
we will prove in Section \ref{sec:cross}. First of all, the point that
$\sigma$ is independent of $k$ is crucial for our
purposes. Furthermore, without it, for a fixed $k$ the theorem would
follow immediately from a suitable variant of Gromov compactness and
it would be enough to assume that $z^k$ is isolated as a periodic
orbit.

Secondly, fix $k$ and $0<\sigma'<\sigma$, and let $V$ be a small
tubular neighborhood of $\tz^k(S^1)$. Consider a $C^\infty$-small
$k$-periodic in time, non-degenerate perturbation $\tH$ of $kH$ in the
class specified in Section \ref{sec:Fl-complex} and a $C^\infty$-small
generic $k$-periodic perturbation of $J$. The 1-periodic orbit $\tz^k$
of $kH$ splits into several non-degenerate periodic orbits of $\tH$
contained in $V$. It follows again from a suitable version of the
Gromov compactness theorem (see, e.g., \cite{Fi}) that every Floer
cylinder of $\tH$ asymptotic to any of these orbits at either end has
energy greater than $\sigma'$. Thus, replacing $\sigma$ by, say,
$\sigma'=\sigma/2$, we obtain the following result.

\begin{Corollary}
  \label{cor:CE}
  Let $z$, $H$ and $J$ be as in Theorem \ref{thm:CE}. Assume that all
  closed Reeb orbits of $\alpha$ are isolated. Then there exists
  $\sigma>0$ such that for all $k\in\N$ and for any choice of
  auxiliary data every arrow to or from $\tz^k$ in the Floer
  graph of $kH$ has length greater than $\sigma$.
\end{Corollary}

\subsection{Vanishing of symplectic homology}
\label{sec:def+general}
The result we need here translates vanishing of symplectic homology to
a certain quantitative property of the filtered Floer homology of the
sequence $kH$. More specifically, it asserts that the homology
$\HF^I(kH)$ is uniformly unstable, i.e., there is a uniform upper
bound on the longest bar, over larger and larger range of action as
$k\to \infty$, whenever $\SH(\alpha)=0$.

\begin{Theorem}
  \label{thm:vanishing}
  Assume that $\SH(\alpha)=0$. Fix $b>0$ and let $H$ be a
  semi-admissible Hamiltonian with $a:=\slope(H)>b$. Then there exists
  a constant $\Cbar>0$ depending only on $H$ such that for any
  interval $I\subset (-\infty, \, kb]$ and every sufficiently large
  $k\in [1,\,\infty)$ the inclusion/quotient map
$$
\HF^{I}(kH)\to \HF^{I+\Cbar}(kH)
$$
is zero. In particular, every bar ending below $kb$ has length less
than $\Cbar$.
\end{Theorem}

In particular, this theorem applies when $\alpha$ is a contact form on
$M\cong S^{2n-1}$ supporting the standard contact structure. Indeed,
in this case, we can think of $M$ as the boundary of a star-shaped
domain $W\subset\WW=\R^{2n}$; and $\SH(\alpha)=0$ since $W$ is
displaceable in $\WW$; see Section \ref{sec:SH}.

\begin{Remark}
  \label{rmk:GS}
  This theorem is closely related to \cite[Prop.\ 3.5]{GS}, which is
  essentially due to Kei Irie, asserting that for some constant $C>0$
  depending only on $\alpha$ the map
  $\SH^I(\alpha)\to \SH^{I+C}(\alpha)$ is zero for any interval $I$ if
  and only if $\SH(\alpha)=0$. In this paper, we need
      Theorem \ref{thm:vanishing} since we work at the level of
      Hamiltonians. The proof of Theorem \ref{thm:vanishing} follows
  roughly the same path as the proof of that proposition with natural
  complications arising from the argument having to deal with a
  specific Hamiltonian rather than a direct limit. We also note that
  Theorem \ref{thm:vanishing} is considerably stronger than Lemma
  \ref{lem:vanishing}. The key difference is that the constant $\Cbar$
  here can be taken independent of $I\subset (-\infty, \, kb]$ and
  $k$, as long as $k$ is sufficiently large. Moreover, $k$ need not be
  an integer here.
\end{Remark}

\begin{Remark}
  As stated, the theorem does not hold for admissible
  Hamiltonians. For the sequence $kH$ is then decreasing on $W$ and
  one has to account for this fact by adding a constant, linearly
  increasing with $k$, to $kH$ or to $\Cbar$. Furthermore, the
  requirement that $I\subset (-\infty, \, kb]$ is essential; for in
  general $\HF(kH)\neq 0$ even though $\SH(\alpha)=0$.
\end{Remark}

\begin{proof} Let $a=\slope(H)$. By the long exact sequence of
    Floer homology groups induced by the inclusion of action
    intervals, it suffices to prove the theorem for a semi-infinite
  interval, which we denote by
  $$
  (-\infty, \, \beta]:=I \subset (-\infty, \, kb].
  $$
  Hence our goal is to show that the map
$$
\HF^{\beta}(kH)\to \HF^{\beta+\Cbar}(kH)
$$
is identically zero for all large $k$ and $\beta\leq kb$ and some
constant $\Cbar>0$.  We carry out the argument in three steps.

\medskip\noindent\emph{Step 1.} Let $\eps>0$ and $\delta>0$ be
sufficiently small. Then, by \eqref{eq:HFlow1} and \eqref{eq:HFlow2}
$\HF^\delta_*(\eps H)\cong \H_{*+n}(W,\p W)$ and, in particular,
$\HF^\delta_{n}(\eps H)\cong \F$. On the other hand, by Lemma
\ref{lem:vanishing}, the natural inclusion/homotopy map
$\iota\colon \HF^\delta_n(\eps H)\to \HF^C_n(\lambda H)$ is
identically zero for some $C>0$ and a sufficiently large $\lambda$
since $\SH(\alpha)=0$.

Next, consider the commutative diagram
\begin{center}
\begin{tikzpicture}
\node (a1) {$\HF^{\beta}(kH)\otimes\HF^{\delta}_{n}(\eps H)$};
\node [right of=a1, xshift=45mm] (a2)
{$\HF^{\beta+\delta}\big((k+\eps)H\big)$};
\node [below of=a1, yshift=-10mm] (b1)
{$\HF^{\beta}(kH)\otimes\HF^{C}_{n}(\lambda H)$};
\node [right of=b1, xshift=45mm] (b2)
{$\HF^{\beta+C}\big((k+\lambda)H\big)$};

\path[draw, -latex'](a1) -- node[yshift=3mm]{} (a2);
\path[draw, -latex'](b1) -- node[yshift=3mm]{} (b2);
\path[draw, -latex'](a1) -- node[xshift=-5mm,
yshift=1mm]{$\id\otimes \iota$} (b1);
\path[draw, -latex'](a2) -- node[yshift=1mm, xshift=2mm]{$\Psi$} (b2);
\end{tikzpicture}
\end{center}
where $\beta \notin \Ss(kH)$ and $\delta>0$ is small enough so that $[\beta, \beta+\delta] \cap \Ss(kH) =\emptyset$.
Here $\Psi$ is induced by a monotone increasing homotopy of the
Hamiltonians and the inclusion of action intervals, and the horizontal
arrows are given by the pair-of-pants product. The top horizontal
arrow in this diagram is an isomorphism since the generator of
$\HF^\delta_{n}(\eps H)\cong \F$ is a ``unit'' with respect to the
pair-of-pants product; see Section \ref{sec:Floer-cont}. Therefore,
$\Psi=0$ since $\iota=0$. Furthermore,
$$
\HF^{\beta+\delta}\big((k+\eps)H\big)\cong \HF^{\beta}(kH)
$$
when $\delta>0$ and $\eps>0$ are sufficiently small. Composing this
isomorphism with $\Psi$, we conclude that the map
$$
\Phi\colon \HF^{\beta}(kH)\to \HF^{\beta+C}\big((k+\lambda)H\big),
$$
again induced by a monotone increasing homotopy of the Hamiltonians
and the inclusion of action intervals, is also zero. This completes
the first step of the proof.

\begin{Remark}
  Continuing Remark \ref{rmk:GS}, note that to prove \cite[Prop.\
  3.5]{GS} we could simply apply this argument to an admissible
  Hamiltonian $H_k$ in place of $kH$ from a cofinal sequence and then
  pass to the limit. However, for the iterates of a specific
  Hamiltonian and an unbounded action range an additional reasoning
  is needed. This is in essence the question of uniform convergence
  and interchanging the limits; cf.\ Remark \ref{rmk:GS}.
\end{Remark}

\noindent\emph{Step 2.}
By Proposition \ref{prop:f}, a monotone increasing homotopy from $kH$
to $(k+\lambda)H$ induces an isomorphism in the filtered homology
$$
\HF^{\tau}(kH)\to \HF^{f(\tau)}\big((k+\lambda)H\big).
$$
Here, as in Section \ref{sec:invariance}, 
$$
f\colon [0,\, A_{kH}(r_{\max})]=[0,\, kA_{H}(r_{\max})] \to [0,\,
A_{(k+\lambda)H}(r_{\max})]=[0,\, (k+\lambda)A_H(r_{\max})] 
$$
is the function $\fa_{(k+\lambda)H}\circ \fa_{kH}^{-1}$. We claim that
\begin{equation}
  \label{eq:f}
\tau\geq f(\tau) \geq \tau  - \lambda h(r_{\max}).
\end{equation}
The exact value of the right-hand side in \eqref{eq:f} is inessential
for our purposes. However, it is important that the right-hand side,
in contrast with the function $f$, is independent of~$k$.

Postponing the proof of the claim, let us finish the proof of the
theorem. First observe that $\beta+C$ is in the range of $f$, i.e.,
$$
\beta+C\leq f\big(kA_H(r_{\max})\big)
$$
when $k$ is large enough.  Indeed, by \eqref{eq:maxAH} with $H$
replaced by $kH$ and $a$ replaced by $ka$, we have
$$
kA_H(r_{\max})\geq ka.
$$
Also recall that $\beta\leq kb$. Thus, by \eqref{eq:f}, $\beta+C$ is
in the range of $f$ whenever
$$
kb + C\leq ka-\lambda h(\rmax).
$$
This condition is automatically satisfied when $k$ is sufficiently
large since $b<a$.

Furthermore, since $f$ is monotone increasing, applying \eqref{eq:f}
to $\tau=\beta+C+\lambda h(\rmax)$, we see that
\begin{equation}
  \label{eq:Cbar}
\beta+C+ \lambda h(\rmax)\geq f^{-1}(\beta+C).
\end{equation}

Consider now the sequence of maps
$$
\HF^{\beta}(kH)\stackrel{\Phi=0}
{\longrightarrow}\HF^{\beta+C}\big((k+\lambda)H\big)\to
\HF^{f^{-1}(\beta+C)}(kH)\to \HF^{\beta+C+ \lambda h(\rmax)}(kH).
$$
Here $f^{-1}(\beta+C)\geq \beta+C>\beta$ and, by Proposition
\ref{prop:f}, the composition of the first two maps is the
inclusion-induced map
$$
\HF^{\beta}(kH)\to
\HF^{f^{-1}(\beta+C)}(kH),
$$
which is then identically zero, for $\Phi=0$. The action level
$f^{-1}(\beta+C)$ depends on $k$, but \eqref{eq:Cbar} provides an
upper bound which does not. Hence, setting
$$
\Cbar=C+\lambda h(\rmax),
$$
we conclude that the map
$$
\HF^{\beta}(kH)\to \HF^{\beta+\Cbar}(kH)
$$
is also zero.

\medskip\noindent\emph{Step 3.} To complete the argument, it remains
to prove \eqref{eq:f}.  Consider the linear monotone increasing
homotopy $F_s:=(k+s\lambda)H$, $s\in [0,1]$, from $F_0=kH$ to
$F_1=(k+\lambda)H$. (Strictly speaking, we should replace here $s$ by
a monotone increasing function on $\R$ equal to $0$ near $-\infty$ and
1 near $+\infty$. However, this technicality does not affect the
result of the calculation.) Setting
\begin{equation}
\label{eq:fs}
f_s:=\fa_{F_s}\circ \fa_{F_0}^{-1},
\end{equation}
we need to bound the change of $f_s(\tau)$ as $s$ ranges from 0 to 1.

Recall from \eqref{eq:fa} that $\fa_{F_s}(T)=A_{F_s}\big(r(s)\big)$,
where
\begin{equation}
  \label{eq:rs-T}
(k+s\lambda)h'(r)=T.
\end{equation}
(In other words, for a closed Reeb orbit with period $T$, the
corresponding 1-periodic orbit of $F_s$ lies on the level $r=r(s)$.)
We also note that the domain of $\fa_{F_s}$ changes with $s$
increasing from $[0,\,ka]$ for $s=0$ to $[0,\,(k+\lambda)a]$ for
$s=1$. Below we always assume that $T$ is in the range $[0,\,ka]$ of
$\fa_{F_0}^{-1}$.

Differentiating \eqref{eq:rs-T} with respect to $s$, we have
\begin{equation}
  \label{eq:action(s)}
\lambda h'(r)+(k+s\lambda)h''(r)r'=0.
\end{equation}
Furthermore, recall that by \eqref{eq:AH} 
$$
\fa_{F_s}(T) = A_{F_s}(r)=(k+s\lambda) \big(r h'(r)-h(r)\big).
$$
Differentiating again, we see that
\begin{equation*}
  \begin{split}
    \frac{d}{ds}\fa_{F_s}(T)
    &=\lambda \big(r h'(r)-h(r)\big)\\
    &\quad +(k+s\lambda) \big(r' h'(r)+ rh''(r)r' -
    h'(r)r')\big)\\
    &=\lambda \big(rh'(r)-h(r)\big)+(k+s\lambda)r h''(r)r'.
  \end{split}
\end{equation*}
By \eqref{eq:action(s)}, the last term in this expression is
$-\lambda r h'(r)$, and hence

\begin{equation*}
  \begin{split}
    \frac{d}{ds}\fa_{F_s}(T)&=\lambda
    \big(rh'(r)-h(r)\big)+(k+s\lambda)r h''(r)r'\\
    &=\lambda \big(rh'(r)-h(r)\big) -\lambda r h'(r)\\
    &=-\lambda h(r). 
  \end{split}
\end{equation*}
Thus this derivative is non-positive and
$$
\left|\frac{d}{ds}\fa_{F_s}(T)\right|\leq \lambda h(r_{\max})
$$
for all $T\in [0,\, a]$. Therefore, by \eqref{eq:fs},
$$
0\geq \frac{d f_s}{ds}=\frac{d\fa_{F_s}}{ds}\circ
\fa_{F_0}^{-1}\geq -\lambda h(\rmax)
$$
on $[0,\, kA_H(\rmax)]$. Integrating with respect to $s$ and using the
fact that $f_0(\tau)=\tau$ and $f_1=f$, we obtain \eqref{eq:f}.
\end{proof}

\subsection{Index recurrence}
\label{sec:index-rec}
Index recurrence has a very different flavor from the first two
results of this section and is essentially a symplectic linear algebra
phenomenon.  Let, as in Section \ref{sec:CZ}, $\mu_\pm$ be the upper
and lower semicontinuous extensions of the Conley--Zehnder index to
the universal covering $\TSp(2m)$.  We need to introduce some further
invariants of $\Phi$, playing a central role in the index recurrence
theorem. Abusing notation, we will automatically extend all invariants
from $\Sp(2m)$ to $\TSp(2m)$ by applying them to the end-point.

Consider first a totally degenerate symplectic linear map
$A\in\Sp(2m)$. (In other words, $A$ is unipotent: all eigenvalues of
$A$ are equal to 1.)  Then we can write $A$ as $A=\exp(JQ)$, where $Q$
is symmetric and all eigenvalues of $JQ$ are equal to 0; see, e.g.,
the proof of \cite[Lemma 4.2]{GG:LS}.  We will view $Q$ as a quadratic
form.  It can be symplectically decomposed into a sum of terms of four
types:
\begin{itemize}
\item the identically zero quadratic form on $\R^{2\nu_0}$,
\item the quadratic form $Q_0=p_1q_2+p_2q_3+\cdots+p_{d-1}q_d$ in
  Darboux coordinates on $\R^{2d}$, where $d\geq 1$ is odd,
\item the quadratic forms $Q_\pm=\pm(Q_0+p^2_d/2)$ on $\R^{2d}$ for
  any $d$.
\end{itemize}
(This variant of the Williamson normal forms is taken from
\cite[Sect.\ 2.4]{AG}.) Clearly, $\dim\ker Q_0=2$ and
$\dim\ker Q_\pm=1$.  Let $b_*(Q)$, where $*=0,\pm$, be the number of
the $Q_0$ and $Q_\pm$ terms in the decomposition. Let us also set
$b_*(A):=b_*(Q)$ and $\nu_0(A):=\nu_0(Q)$.  These are symplectic
invariants of $Q$ and $A$, and the geometric multiplicity $\nug$ of
the eigenvalue $1$ is
$$
\nug=2(b_0+\nu_0)+b_++b_-.
$$
It is not hard to show (cf.\ \cite[Lemma 4.2]{GG:LS}) that for a path
$\Phi$ with $\Phi(1)=A$, we have
\begin{equation}
\label{eq:bpm}
\mu_+(\Phi)=\hmu(\Phi)+b_0+b_+ +\nu_0
\quad\textrm{and}\quad
\mu_-(\Phi)=\hmu(\Phi)-b_0-b_- -\nu_0.
\end{equation}

These formulas readily extend to all paths. Namely, every
$\Phi\in\TSp(2m)$ can be written (non-uniquely) as a product of a loop
$\varphi$ and the direct sum $\Psi_0\oplus \Psi_1$ where
$\Psi_0\in\TSp(2m_0)$ is a totally degenerate (for all $t$) path
$\Psi_0(t)=\exp(JQt)$ and $\Psi_1\in\Sp(2m_1)$ is
non-degenerate. As a consequence, $\mu(\Psi_0)=0$,
$m_0=\nu(\Phi)$ and $m_0+m_1=m$. (Note that $\varphi$ can be absorbed
into $\Psi_1$ unless $m_1=0$.)

Then we set
$$
b_*(\Phi):=b_*(\Psi_0)\textrm{ for $*=0,\pm$ and }
\nu_0(\Phi):=\nu_0(\Psi_0).
$$
These are symplectic invariants of $\Phi$, and
$$
\mu_+(\Phi)=
\hmu(\varphi)+\mu(\Psi_1)+b_0+b_+ +\nu_0
$$
and
$$
\mu_-(\Phi)=\hmu(\varphi)+\mu(\Psi_1)-b_0-b_- -\nu_0.
$$
Let also $2\nua(\Phi)$ be the algebraic multiplicity of the eigenvalue
1 of $\Phi(1)$. Clearly, $\nug/2\leq \nua\leq m$ and
$$
\nua\geq \nu_0+b_0+b_++b_-\geq b_+-b_-.
$$

\begin{Theorem}[Index Recurrence, \cite{GG:LS}]
\label{thm:IRT}
Let $\Phi_0,\ldots,\Phi_q$ be a finite collection of elements in
$\TSp(2m)$ with $\hmu(\Phi_i)>0$ for all $i$. Then for any $\eta>0$
and any $\ell_0\in\N$, there exists an integer sequence $d_s\to\infty$
and $q$ integer sequences $k_{is}$, $i=0,\ldots, q$, going to infinity
as $s\to\infty$ and such that for all $i$ and $s$, and all $\ell\in\N$
in the range $1\leq \ell\leq \ell_0$, we have
\begin{itemize}
\item[\reflb{IR1}{\rm{(i)}}] $\big|\hmu(\Phi^{k_{is}}_i)-d_s\big|<\eta$,
\item[\reflb{IR2}{\rm{(ii)}}]
  $\mu_\pm(\Phi^{k_{is}+\ell}_i) = d_s + \mu_\pm(\Phi^\ell_i)$,
\item[\reflb{IR3}{\rm{(iii)}}]
  $\mu_+(\Phi^{k_{is}-\ell}_i)= d_s -
  \mu_-(\Phi^{\ell}_i)+\big(b_+(\Phi^{\ell}_i)-b_-(\Phi^{\ell}_i)\big)$.
\end{itemize}
In particular,
$\mu_+(\Phi^{k_{is}-\ell}_i)\leq d_s -
\mu_-(\Phi^{\ell}_i)+\nua(\Phi^{\ell}_i)$, and
$\mu(\Phi^{k_{is}-\ell}_i)= d_s - \mu(\Phi^{\ell}_i)$ when $\Phi_i$ is
non-degenerate. Furthermore, for any $N\in \N$ we can make all $d_s$
and $k_{is}$ divisible by~$N$.
\end{Theorem}

This theorem asserts, in particular, that every arbitrarily long
segment starting at $\mu_\pm(\Phi_i)$ and ending at
$\mu_\pm(\Phi_i^{\ell_0})$ will reoccur infinitely many times in the
sequence $\mu_\pm(\Phi_i^k)$, $k\in\N$, up to a common shift
independent of $i$. (Hence the name of the theorem.) When
the paths are non-degenerate we can take the symmetric segment
$\mu(\Phi_i^{-\ell_0}),\ldots, \mu(\Phi_i^{\ell_0})$ with
$\mu(\Phi^0)$ omitted. We also note that when $\Phi$ is dynamically
convex, i.e., $\mu_-(\Phi_i)\geq m+2$, we have
$\mu_+(\Phi^{k_{is}-\ell}_i)\leq d_s-2$ in \ref{IR3}.

Theorem \ref{thm:IRT} can be easily derived from the Common
Index Jump Theorem from \cite{Lo,LZ} or proved
independently. We refer the reader to \cite[Sec.\ 5]{GG:LS} for a
direct proof of a more precise result than stated here. Note that for
one non-degenerate path $\Phi:=\Phi_0$ the assertion readily follows
from Kronecker's theorem -- it is enough to pick the iterations
$k_s:=k_{0s}$ so that all elliptic eigenvalues of $\Phi^{k_s}$ are
sufficiently close to 1. In a similar vein, the more general case of
several non-degenerate paths ultimately relies on Minkowski's theorem
on simultaneous homogeneous approximations.

\section{Proofs of the main theorems}
\label{sec:main-pf}

In this section we prove Theorems \ref{thm:main1} and
\ref{thm:main2}. When $2n-1=1$, Theorem \ref{thm:main1} holds
trivially. Hence, we can assume in both of the proofs that
$2n-1\geq 3$.

\subsection{Proof of Theorem \ref{thm:main1}: Hyperbolic periodic
  orbits and multiplicity}
\label{sec:mult}
We argue by contradiction. Assume that the Reeb flow has only finitely
many simple periodic orbits with positive mean index. We denote these
orbits by $x_0, x_1,\ldots, x_q$. Let $T_i$ be the period of
$x_i$. Without loss of generality we can assume that $z=x_0$ is a
hyperbolic orbit with $\hmu(z)>0$ and $\mu(z)\geq 3$.

In the notation from Section \ref{sec:conv}, let $H$ be a
semi-admissible Hamiltonian such that the slope $a$ of $H$ is greater
than $\max T_i$, where $T_i$ is the period of $x_i$. In particular, $T_0<a$. The $1$-periodic orbits of $kH$
have the form
$$
\tx_i^j(k):=(x_i^j, r_{ij}(k)),
$$
where $x_i^j$ is the $j$-th iterate of $x_i$ and
$r_{ij}(k)\in (0,\infty)$ is determined by the condition
\begin{equation}
  \label{eq:rijk}
kh'\big(r_{ij}(k)\big)=jT_i;
\end{equation}
see \eqref{eq:level}. Note that for every fixed $i$ the range of $j$
is $\big[1, \lfloor ka/T_i \rfloor \big]$. In particular, it is always
non-empty since $T_i <a$; the range grows linearly with $k$ and can be
larger than $k$.

For technical reasons it will also be useful to require that
\begin{equation}
  \label{eq:a2}
\max_i \hmu(z) T_i/\hmu(x_i) < a .
\end{equation}
Finally, we also need to take $a$ sufficiently large so that Theorem
\ref{thm:vanishing} applies with $b$ to be specified later; see
\eqref{eq:b}.

When the iteration order $k$ of the Hamiltonian is clear from the
context, we will suppress it in the notation and simply write
$\tx_i^j$ for $\tx_i^j(k)$. For $i=0$ we will denote the corresponding
orbit by $\tz^j$.  Note that, in spite of the notation, in general
$\tx_i^j$ is not the $j$-th iterate of $\tx_i^1$. However, an
exception is the case of $j=k$. Namely, by \eqref{eq:rijk},
$r_{ik}(k)=r_{i1}(1)$, and hence
$$
\tx_i^k(k)= (x_i^k, r_{i1}(1)).
$$
Denoting $r_*=r_{01}(1)$, we have,  in particular,
\begin{equation}
  \label{eq:r*}
\tz^k(k)=(z^k, r_*), \text{ where } h'(r_*)=T_0.
\end{equation}
The orbit $\tz^k(k)$ gives rise to two generators $\cz^k$ and $\hz^k$
of degrees $|\cz^k|=\mu(z^k)$ and $|\hz^k|=\mu(z^k)+1$ respectively in
the Floer complex of $kH$.

Throughout the proof we will need to make sure that the pair $(H,r_*)$
meets the conditions of Theorem \ref{thm:CE}. It is easy to show that
such Hamiltonians exist. For instance, let us start with $H$ such
that, in addition,
$$
h'''\geq 0\textrm{ on } [1,\, \rho]
$$
for some $\rho>1$. At this point we do not necessary have
$r_*\le \rho$. Consider, however, the family of Hamiltonians
$\lambda H$ with $\lambda\geq 1$. For each of these Hamiltonians
$r_*=r_*(\lambda)$ is determined by
$$
h'(r_*)=T_0/\lambda,
$$
and hence $r_*\to 1$ as $\lambda\to \infty$. For $\lambda$ large
enough, we have $r_*<\rho$. Replacing $H$ by $\lambda H$ and keeping
the notation, we have a semi-admissible Hamiltonian $H$ meeting the
requirements of Theorem \ref{thm:CE}.  Thus $E(u)>\sigma$ for any
non-constant Floer cylinder for $kH$ asymptotic to $(z^k,r^*)$ at
either end, where the constant $\sigma>0$ is independent of $k$ and
$u$.

Fix 
\begin{equation}
   \label{eq:b}
  b> \CA_H(\tz)= r_*T_0-h(r_*).
 \end{equation}
 We will further require that $a=\slope(H)>b$, and hence Theorem
 \ref{thm:vanishing} applies.  Again, it is easy to see that $H$
 meeting all of the above requirements exist. Let $\Cbar>0$ be as in
 that theorem. In other words, the map
$$
\HF^{I}(kH)\to \HF^{I+\Cbar}(kH)
$$
vanishes for large $k$ and any $I \subset (-\infty,\, kb]$.

In the proof, we will apply the Index Recurrence Theorem (Theorem
\ref{thm:IRT}) with $m=n-1$ to the linearized Reeb flows $\Phi_i$
along $x_i$, including the flow along $z=x_0$. In the theorem, we can
take any $\eta<\min\{\sigma/C, 1/2\}$, where the constant $C$ is to be
specified later (see \eqref{eq:C}), and any $\ell_0\in \N$ with
\begin{equation}
  \label{eq:ell0}
  \ell_0>\frac{n+3}{\min_{i\geq 1} \hmu(x_i)}.
\end{equation}
We will need to have $s$ sufficiently large so that $d_s$ and $k_{is}$
are also large. For the sake of brevity, suppressing $s$ in the
notation, we set $d=d_s$, $k=k_{0s}$ and $k_i=k_{is}$ for
$i=1,\ldots,q$.

The key to the proof is the following result.

\begin{Lemma}
  \label{lemma:arrows}
  The length of the shortest arrow to or from $\hz^k$ in the
  Floer complex of $kH$ goes to infinity as $s\to \infty$, and hence
  $k\to\infty$.
\end{Lemma}

The theorem readily follows from the lemma. Indeed, let
$A=\CA_{kH}(\tz^k)$ and $I=[A-2\Cbar,\, A+2\Cbar]$. Then $\hz^k$ is
closed in $\CF^I (kH)$ when $k$ is large enough. Furthermore,
$[\hz^k]\neq 0$ in $\HF^I (kH)$ and its image in $\HF^{I+\Cbar} (kH)$
is also non-zero. This contradicts Theorem
\ref{thm:vanishing}. (Alternatively, one can use here a variant of
\cite[Prop.\ 3.8]{CGG:Entropy}.)

\begin{proof}[Proof of Lemma \ref{lemma:arrows}]
  Clearly, $\hz^k$ is not connected by an arrow to any orbit $y$ with
  $\hmu(y)\leq 0$ or any generator arising from $W$ when $k$ is large;
  see Section \ref{sec:Fl-complex}. Hence, to prove the lemma, it
  suffices to show that the ``index distance'' from $\hz^k(k)$ to
  $\supp \HF\big(\tx^j_i(k)\big)$ is greater than or equal to 2 or the
  action difference between the action of $\tz^k(k)$ and $\tx^j_i(k)$
  becomes arbitrarily large as $k\to\infty$ or is smaller than
  $\sigma$. (We emphasize that all 1-periodic orbits here are taken
  for the same Hamiltonian $kH$ where $k=k_{0s}$. However, we usually
  suppress this role of $k$ in the notation. Thus $\tz^k:=\tz^k(k)$,
  $\hz^j:=\hz^j(k)$, $\tx^j_i:=\tx^j_i(k)$, etc.)
  
  First, note that since $z$ is hyperbolic, the mean index of $z^j$ is
  equal to its Conley--Zehnder index:
  \begin{equation}
    \label{eq:index-mindex}
\hmu(z^j)=j\hmu(z)=\mu(z^j)= j\mu(z)\in \Z.
\end{equation}
In particular, by Condition \ref{IR1} of Theorem
\ref{thm:IRT},
$$
d=\hmu(z^k)=k\hmu(z)=\mu(z^k)
$$
since $0<\eta<1/2$, and hence
$$
|\cz^k|=d \textrm{ and } |\hz^k|=d+1.
$$

For the sake of brevity, we will write $y:=x^j_i$ and $x=x_i$. Setting
$$
l =j-k_i,
$$
we will treat separately several cases determined by the value of $l$
and whether or not $z=x$ (i.e., $i=0$) or $z\neq x$ (i.e., $i\geq 1$).

\medskip 

\noindent\emph{Case 0: $z=x$, i.e., $i=0$.} Clearly, the orbits
$\cz^k$ and $\hz^k$ are not connected by a Floer arrow even though the
index difference is one. Next, due to the requirement that
$\mu(z)\geq 3$ and \eqref{eq:index-mindex}, for $j\neq k$ the degree
difference between $\cz^k$ or $\hz^k$ and $\cz^j$ or $\hz^j$ is at
least two. Hence, neither of the orbits $\cz^k$ or $\hz^k$ is
connected by a Floer arrow to either $\cz^j$ or $\hz^j$. From now on
we will assume that $i\geq 1$.

\medskip 

\noindent\emph{Case 1: $l=0$, i.e., $j=k_i$, and $i\geq 1$.}
This case is central to the proof of the theorem and the argument
amounts to controlling the action difference.

Recall that $k=k_{0s}$ and $j=k_{is}$ depend on the parameter $s$, and
$k\to\infty$ and $j\to\infty$ as $s\to\infty$. By Condition \ref{IR1}
of Theorem \ref{thm:IRT}, we have
$$
\frac{j}{k} \to \frac{\hmu(z)}{\hmu(x)} \quad \text{as} \quad s \to
\infty.
$$

Consider now the periodic orbits $\ty=(x^j,r^*)$ and $\tz^k=(z^k,r_*)$
of $H^{\# k}$, where $r^*:=r_{ij}(k)$. Here $r_*$ is determined by
\eqref{eq:r*} and
$$
kh'(r^*)=jT
$$
by \eqref{eq:rijk}. In other words,
\begin{equation}
  \label{eq:h'r*}
  h'(r^*)=\frac{j}{k} T \to
  \frac{\hmu(z)}{\hmu(x)}T
  \quad \text{as} \quad s\to\infty.
\end{equation}

We will break down the argument into two subcases: either
\begin{equation}
  \label{eq:2groups}
\frac{T}{\hmu(x)}=\frac{T_0}{\hmu(z)},
\end{equation}
or these two ratios are distinct. Here $T=T_i$ is the Reeb action of
$x_i$ and $T_0$ is the Reeb action of $z$.

We will show that for a large $s$, and hence $k$ and $j$,
\eqref{eq:2groups} results in a short action gap (smaller than
$\sigma$) between 1-periodic orbits $\ty$ and $\tz^k$ of $kH$ and the
inequality leads to a large action gap going to infinity as
$k\to \infty$. By rescaling $\alpha$, we can assume without loss of
generality that
\begin{equation}
  \label{eq:T0hmu}
T_0=\hmu(z) =d.
\end{equation}

Let us first focus on the former subcase -- the equality,
\eqref{eq:2groups}. Then we also have
\begin{equation}
  \label{eq:Thmu}
T=\hmu(x).
\end{equation}
Also note that the limit in \eqref{eq:h'r*} is then equal to $T_0$.

Let
$$
(h')^{\inv}\colon [0,\, a]\to [1,\, r_{\max}]
$$
be the inverse function of $h'$, where $a$ is the slope of $H$. This
is a continuous function, which is $C^\infty$-smooth on $(0,\,
a)$. (At the end-points $0$ and $a$ the derivative of this function is
$+\infty$.) Set
$$
C_1:= 2\,\frac{d (h')^\inv }{d\tau}(T_0) <\infty.
$$
Then, since $jT/k\to T_0$, when $k$ is large enough we have
$$
|r_*-r^*|\leq C_1 | T_0-jT/k |.
$$

Set
$$
C_2=\max_{[1,\, r_{\max}]} |A_h'|
= \max_{[1,\, r_{\max}]} r|h''(r)|<\infty.
$$
Then
\begin{equation*}
\begin{split}
  \big|\CA_{kH}(\ty)-\CA_{kH}(\tz^k)\big|
      &= \big|k A_h(r^*)-k A_h(r_*)\big| \\
      & \leq C_2 k | r^*-r_* | \\
      & \leq C_1C_2 k | T_0-jT/k | \\
      & = C_1C_2  | kT_0-jT | \\
      & = C_1C_2  | k\hmu(z)-j\hmu(x)| \\
      & \leq C_1C_2  \eta .
      \end{split}
    \end{equation*}
    Here we used Condition \ref{IR1} of Theorem \ref{thm:IRT} and the
    convention \eqref{eq:T0hmu}. Without this convention, the upper
    bound would be $\big(C_1C_2T_0/\hmu(z)\big)\eta$.

    Therefore,
$$
\big|\CA_{kH}(\ty)-\CA_{kH}(\tz^k)\big|<\sigma,
$$
when
\begin{equation}
  \label{eq:C}
C \eta<\sigma, \text{ where } C:=C_1C_2T_0/\hmu(z),
\end{equation}
and $\ty$ and $\tz^k$ are not connected by a Floer arrow by the
Crossing Energy Theorem (Theorem \ref{thm:CE}).

Next, let us focus on the subcase where \eqref{eq:2groups}
fails. Again, without loss of generality we can require
\eqref{eq:T0hmu}. Now, in place of \eqref{eq:Thmu}, we have
$$
|T-\hmu(x)|>\delta
$$ for some $\delta>0$.

Set
$$
c_1:=\min_{[0,\, a]}\left|\frac{d (h')^\inv (\tau)}{d\tau}\right|>0.
$$
Let $\xi>0$ be so small that
$$
1+\xi \leq (h')^{\inv}\big(\hmu(z)T/\hmu(x)\big)\leq \rmax-\xi
$$
and
$$
1+\xi \leq r_*=(h')^{\inv}(T_0)\leq \rmax-\xi.
$$
Such $\xi$ exists due to \eqref{eq:a2}. Then, by \eqref{eq:h'r*},
$r^*\in \Gamma$ and, by construction, $r_*\in \Gamma$, where
$\Gamma:=[1+\xi,\, \rmax-\xi]$. Let
$$
c_2:=\min_{\Gamma} |A_h'|\geq  (1+\xi) \min_\Gamma |h''|>0.
$$
We emphasize that both of the constants $c_1$ and $c_2$ are strictly
positive and the interval with end-points $r_*$ and $r^*$ is contained
in $\Gamma$ when $k$ is large. Therefore,
\begin{equation*}
\begin{split}
  \big|\CA_{kH}(\ty)-\CA_{kH}(\tz^k)\big|
      &= \big|k A_h(r^*)-k A_h(r_*)\big| \\
      & \geq c_2 k | r^*-r_* | \\
      & \geq c_1c_2 k | T_0-jT/k | \\
      & = c_1c_2  | kT_0-jT | \\
      & \geq c_1 c_2 \delta j- c_1c_2 |k\hmu(z)-j\hmu(x)| \\
      & \geq c_1 c_2 \delta j-  c_1c_2 \eta .
      \end{split}
    \end{equation*}
    As a consequence, the length of an arrow between $\tz^k$ and
    $\ty$, if it exists, is bounded from below by
    $c_1 c_2 \delta j- c_1c_2 \eta$. Furthermore,
    $$
    c_1 c_2 \delta j-  c_1c_2  \eta \to \infty
    $$
    as $j\to \infty$ and, hence, as $s \to \infty$.
      This completes the proof of the lemma in Case~1.

    \begin{Remark}
      Clearly condition \eqref{eq:2groups} is extremely
      non-generic. However, we cannot use a small perturbation to
      eliminate it, for the contradiction assumption that the Reeb
      flow of $\alpha$ has finitely many simple periodic orbits is
      already non-generic. In fact, under this assumption, one would
      expect resonance relations of this type to hold; see, e.g.,
      \cite[Sec.\ 6]{GG:LS} and references therein.
    \end{Remark}

    It remains to deal with the situation where
    $$
    l=j - k_i \neq 0
    $$
    and $i\geq 1$, which is handled via the index gap by Theorem
    \ref{thm:IRT}. Note that in the setting of that theorem
    $$
    \ell=| l |. 
    $$
    Recall also that, by \eqref{eq:support},
    $$
    \supp \HF(\ty)\subset \CI:=[\mu_-(y),\, \mu_+(y)+1]\subset
    [j\hmu(x)-n+1,\, j\hmu(x)+n].
$$

 \medskip 

 \noindent\emph{Case 2: $\ell:=|l| >\ell_0$ and $i\geq 1$.}
 Let us show that neither of the orbits $\cz^k$ or $\hz^k$ is
 connected by a Floer arrow to $\ty$. Assume first that $l$ is
 positive, i.e., $l>\ell_0$. Then
$$
\CI\subset k_i\hmu(x)+ [l\hmu(x)-n+1,\, l\hmu(x)+n]\subset d+
[l\hmu(x)-n,\, l\hmu(x)+n+1],
$$
where the second inclusion relies on the requirement that $0<\eta<1$
together with the fact that $|k_i\hmu(x)-d|<\eta$ by Condition
\ref{IR1} from Theorem \ref{thm:IRT}. Thus the distance from
$|\cz^k|=d$ or $|\hz^k|=d+1$ to $\CI$, and hence to $\supp \HF(\ty)$,
is at least 2; for $l\hmu(x)-n \geq \ell_0\hmu(x) -n\geq 2$ by
\eqref{eq:ell0}. When $l<-\ell_0$, the argument is similar.

 \medskip 

 \noindent\emph{Case 3: $\ell:=|l| \leq \ell_0$ and $i\geq 1$.}
 Arguing as in Example \ref{ex:no-arrows}, we claim that under the
 assumptions of the theorem, the index difference between
 $|\hz^k|=d+1$ and $\supp\HF(\ty)$ is again at least 2, and hence
 $\hz^k$ is not connected to $\ty$ by a Floer arrow. To see this,
 assume first that $l>0$. Then, by Condition \ref{IR2} of Theorem
 \ref{thm:IRT} with $l=\ell$,
$$
\CI\subset [d+\mu_-(x^\ell),\,\infty)\subset [d+3,\,\infty),
$$
and the distance from $|\hz^k|=d+1$ to $\CI$, and hence to
$\supp \HF(\ty)$, is at least 2. When $l<0$, setting $l=-\ell$, we
have
$$
\CI\subset (-\infty, \, d-\mu_-(x^\ell)
+\nua(x^\ell)+1]\subset (-\infty,\, d-1]
$$
by Condition \ref{IR3}. Therefore, the distance from $|\hz^k|=d+1$ to
$\supp \HF(\ty)\subset \CI$ is also at least 2. This completes the
proof of the lemma and the proof of Theorem~\ref{thm:main1}.
\end{proof}

\begin{Remark}
  \label{rmk:alt}
  Note that in Cases 0--2 of the proof, the orbits $\cz^k$ and $\hz^k$
  play similar roles and neither one of them is connected to $\ty$ by
  a (short) arrow. Moreover, no dynamical convexity type condition
  other than that $\mu(z)\geq 3$ was used in these steps.  However, in
  Case 3 the index conditions enter the picture and the symmetry
  between the orbits $\cz^k$ and $\hz^k$ breaks down. We have chosen
  to work with $\hz^k$ and impose the assumption that
  $\mu_-(x)\geq \max\big\{3,\, 2+\nua(x)\big\}$. Alternatively, one
  could have used here $\cz^k$, requiring instead that
  $\mu_-(x) \geq 3 + \nua(x)$ for all closed orbits including
  $x=z$. However, when $\alpha$ is degenerate this condition is more
  restrictive and does not appear to automatically follow from
  dynamical convexity. Furthermore, note that in the statement of
  Theorem \ref{thm:hyperbolic} and also in this modification of the
  theorem, we could have replaced $\nua$ by $b_+-b_-$.
\end{Remark}

\subsection{Proof of Theorem \ref{thm:main2}: Non-existence of locally
  maximal orbits}
\label{sec:pf-loc-max}
The proof follows the same line of reasoning as the proof of Theorem
\ref{thm:hyperbolic} and we will only outline the argument,
emphasizing the necessary modifications and skipping some details. We
will also keep the notation from that proof.

Thus let $z=x_0$ be a non-degenerate, locally maximal orbit of the
Reeb flow and let $x_1,\ldots,x_q$ be the remaining simple closed
orbits. By the dynamical convexity condition, $\hmu(x_i)\geq 2$ for
all $i=0,\ldots,q$.  As in the proof of Theorem \ref{thm:main1}, it is
easy to find an admissible Hamiltonian $H$ such that the Crossing
Energy Theorem (Theorem \ref{thm:CE}) holds for $\tz$. Furthermore,
the Index Recurrence Theorem (Theorem \ref{thm:IRT}) applies to the
orbits $z=x_0,\ldots,x_q$. As before, $k:=k_{0s}$.  Our goal is to
prove an analogue of Lemma \ref{lemma:arrows}. Namely, we will show
that at least one of the orbits $\cz^k$ or $\hz^k$ cannot be connected
by a short arrow to any other closed orbit.

Due to the non-degeneracy requirement, every Reeb orbit $x_i^j$ gives
rise to two generators $\cx_i^j$ and $\hx_i^j$ in $\CF(kH)$ with
$|\cx_i^j|=\mu(x_i^j)$ and $|\hx_i^j|=\mu(x_i^j)+1$.

Recall that, by Condition \ref{IR1} of Theorem \ref{thm:IRT},
$$
|d-\hmu(x_i^{k_{is}})|<\eta
$$
and that, by non-degeneracy,
$$
|\hmu(x_i^{j}) - \mu(x^{j}_i)|<n-1.
$$
for any $j\neq 0$. In particular,
$$
|d-\hmu(z^k)|<\eta
$$
and 
$$
|\mu(z^k) - \hmu(z^k)|<n-1.
$$
Set
$$
\ell_0:=n+1.
$$
Then, similarly to Case 2 of Lemma \ref{lemma:arrows}, when
$\ell:=|j-k_{is}| > \ell_0$, the degree difference between $\cz^k$ or
$\hz^k$ and $\cx_i^j$ or $\hx_i^j$ is greater than or equal to 2.

Next, as in Case 3, assume that $0< \ell=|j-k_{is}| \leq \ell_0$. We
claim that then $\cz^k$ is not connected to $\cx_i^j$ and $\hx_i^j$
when $d<\hmu(z^k)$ and that $\hz^k$ is not connected to either of
these orbits when $d>\hmu(z^k)$.

To see this, note first that when $j>k_{is}$ the range of $|\cx_i^j|$
is $[d+n+1, \, \infty)$ and the range of $|\hx_i^j|$ is
$[d+n+2, \, \infty)$ by Condition \ref{IR2} of Theorem
\ref{thm:IRT}. Likewise, when $j<k_{is}$, the range of $|\cx_i^j|$ is
$(-\infty, \, d-n-1]$ and the range of $|\hx_i^j|$ is
$(-\infty, \, d-n]$ by Condition \ref{IR3}.

When $d<\hmu(z^k)$, the range of $|\cz^k|$ is $[d-n+2, \, d+n-1]$ and
the range of $|\hz^k|$ is $[d-n+3, \, d+n]$. Thus the degree
difference between $\cz^k$ and $\cx_i^j$ or $\hx_i^j$ is at least 2.

When $d>\hmu(z^k)$, the range of $|\cz^k|$ is $[d-n+1, \, d+n-2]$ and
the range of $|\hz^k|$ is $[d-n+2, \, d+n-1]$. Therefore, in this
case, the degree difference between $\hz^k$ and $\cx_i^j$ or $\hx_i^j$
is at least 2.

Note that here we allow $i$ to be 0, and hence we have in particular
shown that $\cz^k$ or $\hz^k$ is not connected by an arrow with
$\cz^j$ and $\hz^j$ for any $j$. (This is an analogue of Case 0 of
Lemma \ref{lemma:arrows}, which is now incorporated into Cases 2 and
3.)

The remaining case is where $j=k_{is}$ and $i\geq 1$. This case is
handled exactly as Case 1 in the proof of Lemma \ref{lemma:arrows}:
the action difference is either less than $\sigma$, which is
impossible by Theorem \ref{thm:CE} applied to $z$, or goes to infinity
as $k\to\infty$, which is also impossible by Theorem
\ref{thm:vanishing}. This contradiction completes the proof of Theorem
\ref{thm:main2}. ~\hfill\qed

\section{Proof of the crossing energy theorem}
\label{sec:cross}
In this section we prove the Crossing Energy Theorem -- Theorem
\ref{thm:CE}. The proof comprises two major components. The first one,
Theorem \ref{thm:location}, deals with the location of Floer
trajectories. Roughly speaking, the goal of this part is to show that
a small energy Floer trajectory cannot enter the region where $r$ is
close to $1$ with a lower bound on $r$ independent of the order of
iteration. This part is completely new. The second component gives
then an energy lower bound, also independent of the order of
iteration. The proof of the second part is quite similar to the
argument in \cite{CGG:Entropy, GG:hyperbolic, GG:PR}.

\subsection{Location constraints}
\label{sec:location}

The key new ingredient in the proof of the Crossing Energy Theorem is
the following result which, under a minor additional condition on $H$,
limits the range of $r\circ u$.

Before stating the result we need to set some new conventions.
Throughout the proof, it is convenient to adopt a different
Hamiltonian iteration procedure (see, for instance,
    Remark \ref{rmk:salamon}).  Namely, passing to the $k$th iterate
of $H$ (or to be more precise of $\varphi_H$), rather than looking at
1-periodic orbits of the Hamiltonian $kH$, we will look at the
$k$-periodic orbits of $H$, changing the time range from $S^1=\R/\Z$
to $S^1_k=\R/k\Z$. We will refer to the resulting Hamiltonian as
$H^{\sharp k}$; cf.\ \cite{GG:hyperbolic, GG:PR}. A time-dependent
almost complex structure is then also assumed to be $k$-periodic in
time rather than 1-periodic. This modification does not affect the
Floer complexes, the action and the action filtration, the energy of a
Floer trajectory, etc., with an isomorphism given by an
$(s,t)$-reparametrization.

Furthermore, whenever $H|_W=\const$ a part of a solution $u$ of the
Floer equation is a (anti-)holomorphic curve. Hence, the standard
monotonicity (see, e.g., \cite{Si}) implies that solutions with small
energy can only enter a narrow collar of $M=\p W$ in $W$. Thus
extending $r$ to the collar as a function to
$(1-\eta,\,1]\subset (0,\,\infty)$, for some small $\eta>0$, as in
Section \ref{sec:setting} we can assume without loss of generality
that the composition $r\circ u\colon \R\times S^1_k \to (0,\,\infty)$
is always defined. By a suitable variant of the Gromov compactness
theorem (see, e.g., \cite{Fi}) the same is true for a $C^\infty$-small
perturbation of $H^{\sharp k}$.

It is convenient throughout the proof to think of $r$ as taking values
in $(0,\,\infty)$ rather than in $[1,\infty)$ or
$[1-\eta,\,\infty)$. Likewise, we extend $h$ to a smooth function on
$(0,\,\infty)$ by setting $h|_{(0,\,1]}\equiv 0$.

In what follows, $H$ is admissible or semi-admissible and $J$ is
admissible. In particular, $H$ and $J$ are independent of time. As
above, we will assume that all Floer cylinders $u$ we consider have
sufficiently small energy and hence are contained
$M\times (1-\eta,\,\infty)$ for some small $\eta>0$.

\begin{Theorem}
\label{thm:location}
Let $H(r,x) =h(r)$ be a semi-admissible Hamiltonian. Assume that
$1<r_*^-\leq r_*^+$ and $\delta >0$ are such that $1<r_*^--\delta$ and
$r_*^+ +\delta < \rmax$, and
\begin{equation}
  \label{eq:h'''2}
  h''' \geq 0 \textrm{ on } [1, \, r_*^++\delta] \subset [1, \, \rmax).
\end{equation}
Fix an admissible almost complex structure $J$. Then there exists
$\sigma >0$ (depending only on $H$, $J$, $r_*^\pm$ and
  $\delta$) such that for any $k \in \N$ and any Floer cylinder
$u \colon \R \times S_k^1 \to \WW$ of $H^{\sharp k}$ with energy
$E(u) \leq \sigma$ and asymptotic, at either end, to a periodic orbit
in $M\times [r_*^-,\,r_*^+]$, the image of $u$ is contained in
$M\times (r_*^--\delta,\, r_*^++\delta)$.
\end{Theorem}

\begin{Remark}
  By the target-local Gromov compactness theorem from \cite{Fi}, the assertion of
  the theorem still holds with perhaps a smaller value of $\sigma$
  when $H^{\sharp k}$ and $J$ are replaced by their compactly
  supported, $k$-periodic in time $C^\infty$-small perturbations.
  (The size of the perturbation may depend on $k$.) Likewise, one can
  use perturbations of the type specified in Section \ref{sec:degen}.
\end{Remark}  

The proof of Theorem \ref{thm:location} is an easy application of the
following technical result. Recall from Section \ref{sec:setting} that
the action function $A_h$ is defined as $A_h(r)=rh'(r)-h(r)$, where
having extended $h$ to $(0,\,\infty)$ as 0 on $(0,\,1]$ we also have
$A_h\equiv 0$ on $(0,\,1]$.

\begin{Theorem}
\label{prop:location}
Let $H(r,x) =h(r)$ be a semi-admissible Hamiltonian and $J$ be an
admissible almost complex structure. There exist $\eps >0$ and $C >0$,
depending only on $H$ and $J$,
such that for any $k \in \N$ and any Floer cylinder
$u \colon \R \times S_k^1 \to \WW$ of $H^{\sharp k}$ with
$E(u)\leq \eps$ and contained in the region where $h''' (r) \geq 0$,
we have
\begin{equation}
\label{eq:prop_loc}
\inf_{\R \times S^1_k} r \big( u (s,t) \big)
\geq r^- - \frac{ C\cdot (r^+)^{3/4}}{\sqrt{A_h(r^+)}} E(u)^{5/8} ,
\end{equation}
where $ r^{\pm} = r \big( u (\mp \infty, t) \big)$ and
$A_h(r^+)=r^+ h' (r^+) - h(r^+)$.
\end{Theorem}

\begin{Remark}
  \label{rmk:salamon}
  Let us specify how $\eps>0$ and $C>0$ in Theorem \ref{prop:location}
  are chosen. By \cite[Sec.\ 1.5]{Sa}, there exist $\eps'>0$ and
  $C'>0$, depending only on $H$ and $J$, such that for any Floer
  cylinder $u\colon \R\times S^1_k\to \WW$, the point-wise upper bound
\label{rmk:salamon}
\begin{equation}
\label{eq:salamon}
\big\| \partial_s u(s, t) \big\| < C' E(u)^{1/4}
\end{equation}
holds whenever $E(u) <\eps'$. (This is an instance when it is more
convenient to work with $H^{\sharp k}$ than $kH$; the proof of
\eqref{eq:salamon} is local in the domain of $u$, but the constants
$\eps'$ and $C'$ depend on $H$ and its first and second
derivatives. Note also that, by the maximum principle,
  \eqref{eq:max_prince}, the image of $u$ is contained in a fixed,
  independent of $k$, compact subset of $\WW$ determined by $H$.) In
the proof of the Theorem \ref{prop:location} we take $\eps = \eps'$
and $C=\sqrt{4C'}$.
\end{Remark}

Let us now prove Theorem \ref{thm:location} assuming Theorem
\ref{prop:location}.

\begin{proof}[Proof of Theorem \ref{thm:location}] 
  Let $\eps > 0$ be as in Theorem \ref{prop:location}. Choose
  $0< \sigma' <\eps$ such that $r^+-r^- <\delta/2$ for any $k \in \N$
  and any Floer cylinder $u$ of $H^{\sharp k}$ with energy
  $E(u) <\sigma'$, which is asymptotic at least at one of the
  ends to a $k$-periodic orbit in $M\times [r_*^-,\,r_*^+]$.
  Such a constant $\sigma'$ exists since $A_h'=r h'' >0$ on
  $(r_*^--\delta,\, r_*^++\delta)$ and $A_h'=r h'' \geq 0$
    everywhere. Note that we can apply Theorem \ref{prop:location} to
  any $u$ as above with $E(u) < \sigma'$. Next we choose
  $0 < \sigma <\sigma'$ such that
  \[
    \sigma^{5/8}<\frac{\delta\sqrt{A_h(r_*^-)}}{2C\cdot(r_*^+)^{3/4}}.
  \]
It follows from \eqref{eq:prop_loc} that
$$
\inf_{\R \times S^1_k} r \big( u (s,t) \big) 
> r^- - \delta/2
$$
whenever $E(u) < \sigma$. Now, using the maximum principle,
\eqref{eq:max_prince}, we conclude that the image of $u$ is contained
in $M\times (r_*^--\delta,\, r_*^++\delta)$ whenever
$E(u) < \sigma$.  This completes the proof Theorem \ref{thm:location}
with $\sigma$ chosen as above.
\end{proof}

\begin{Remark}
  Note that in this argument at least one of the points $r^+$ and
  $r^-$ is in the interval $[r_*^-,\,r_*^+]$ but \emph{a priori} we do
  not know which one. When this is $r^+$, a small energy trajectory
  $u$ is contained in $M\times (r_*^--\delta,\, r_*^+]$.
  However, this is not necessarily the case when
  $r^+\not\in [r_*^-,\,r_*^+]$, e.g., if $r^-=r_*^-=r^+_*<r^+$.
\end{Remark}

\subsection{Proof of Theorem \ref{prop:location}: The $r$-range}
We carry out the proof in four steps. First, we show that the average
in $t\in S^1_k$ of $r\big(u(s,t)\big)$ satisfies the differential
inequality \eqref{eq:step1}. In the second step, we establish lower
and upper bounds on this average.  This is where the third derivative
condition on $h$ becomes essential. The goal of the third step is to
prove inequality \eqref{eq:step3}, which is central to the proof and
gives an upper bound on the amount of time $u(s,t)$ spends below a
certain fixed $r$-level.  Finally, in the last step, we finish the
proof of the theorem by combining the key inequality,
\eqref{eq:step3}, with the pointwise upper bound \eqref{eq:salamon} on
$\| \p_s u\|$.

\subsubsection{Step 1} 
Fix a semi-admissible Hamiltonian $H(x,r)=h(r)$, an iteration
$k \in \N$ and a Floer cylinder $u \colon \R \times S_k^1 \to \WW$ of
$H^{\sharp k}$ as in Theorem \ref{prop:location}. Let us denote by
$u_s:=u(s,\cdot) \colon S_k^1 \to \WW$ the $s$-slice of the cylinder
$u$. In this step we show that
\begin{equation}
\label{eq:step1}
\frac{d}{ds} \int_0^k r (u_s)\, dt \leq - k A_h(r^-)
+ \int_0^k A_h\big(r(u_s)\big)\, dt.
\end{equation}
Using \eqref{eq:complex_strc} and the Floer equation,
\eqref{eq:floer_2}, we have
\begin{align*}
  \partial_s (r (u))
&= dr (\partial_s u)  \\
&= dr \big(J (\partial_t u - X_H(u) ) \big)  \\
&= -r(u) \alpha (\partial_t u) + r(u) \alpha(X_H(u)) \\
&= - r(u) \alpha(\partial_t u) + r(u) h'(r (u))  \\
&= - \big [ r(u) \alpha(\partial_t u) - h(r(u)) \big]
    + \big[ r(u)h'(r (u))-h(r(u)) \big]. 
\end{align*}
Integrating over $S_k^1$ gives
\begin{equation}
\label{eq:step1_2}
\frac{d}{ds} \int_0^k r (u_s) dt =  - \A_{H^{\sharp k}} (u_s)
+ \int_0^k A_h\big(r(u_s)\big)\, dt.
\end{equation}
Recall that by \eqref{eq:floer_1} the action is decreasing along Floer
cylinders. Hence we have
\begin{equation}
\label{eq:step1_3}
\A_{H^{\sharp k}} (u_s)  \geq \A_{H^{\sharp k}} (u_{\infty}) = kA_h(r^-).
\end{equation}
Recall that by our conventions $r^{\pm}= r(u_{\mp \infty})$.
Then the desired inequality, \eqref{eq:step1}, follows from
\eqref{eq:step1_2} and \eqref{eq:step1_3}.

\subsubsection{Step 2} 
From now on we assume that the image of $u$ is contained in a region
where $h'''(r) \geq 0$. In this step we show that
\begin{equation}
\label{eq:step2}
\frac{A_h(r^-)}{A_h(r^+)} k r^+ \leq
\inf_{s\in \R} \int_0^k r (u_s)\, dt  \leq kr^+ .
\end{equation}
The upper bound for the middle term follows from the maximum
principle, \eqref{eq:max_prince}. Let us prove the lower bound.

\begin{Remark}
\label{rmk:step2}
Note that since $E(u) =k (A_h(r^+) -A_h(r^-))$, the lower bound in
\eqref{eq:step2} can also be written as $kr^+ - E(u)r^+/A_h(r^+)$.
\end{Remark}

We will need the following claim, where in the setting of the theorem
we can take $r_0= r_*^+ + \delta$.

\smallskip

\noindent
\textit{Claim 1:} The function $A_h(r)/r$ is non-decreasing on
$[1,\, r_0]$, provided that $h''' \geq 0$ on this interval.

\smallskip

Assuming the claim, let us prove the first inequality
in \eqref{eq:step2}, i.e., 
\begin{equation}
\label{eq:step22}
\frac{A_h(r^-)}{A_h(r^+)} k r^+ \leq
\int_0^k r (u_s)\, dt 
\end{equation}
for all $s\in (-\infty,\, \infty]$. First, note that by the claim
\[
  \frac{A_h(r^-)}{A_h(r^+)} k r^+  \leq kr^-.
\]
Hence \eqref{eq:step22} holds at $s=0$ and  $s=\infty$. It remains to verify
\eqref{eq:step22} at the critical points of the function
$$
s\mapsto \int_0^k r (u_s)\, dt.
$$
Thus let $s_0$ be a critical point. Using \eqref{eq:step1}, the
maximum principle \eqref{eq:max_prince} and the claim above, we have
\[
  k A_h(r^-) \leq \int_0^k A_h\big(r(u_{s_0})\big)\, dt
  \leq \frac{A_h(r^+)}{r^+}  \int_0^k r(u_{s_0})\, dt. 
\]
It follows that \eqref{eq:step22} holds at the critical points as
well. This completes the proof of \eqref{eq:step2}, assuming the
claim.

\begin{proof}[Proof of Claim 1]
  Recall that $A_h(r)=r h'-h$. Dividing by $r$ and differentiating, we
  have
$$
\frac{d}{dr} \big(A_h(r)/r\big)=\frac{r A'_h-A_h}{r^2}
= \frac{r^2 h''-r  h'+ h}{r^2}.
$$
Thus $d\big(A_h(r)/r\big)/dr\geq 0$ whenever
$f=r^2 h''-r h'+ h\geq 0$. Since $h'''\geq 0$, the function $h''$ is
monotone increasing. This together with the fact that $h'(1)=0$
(since $H(x,r) =h(r)$ is semi-admissible) implies that
$$
h'(r) \leq (r-1) h''(r).
$$
 Indeed,
$$
h'(r) = \int_1^r h''(\xi)\, d\xi \leq \int_1^r h''(r)\,d\xi
= (r-1)h''(r).
$$
Hence,
$$
f=r(r h''- h')+ h\geq r\big(rh''-(r-1)h''\big) + h = rh''+h\geq 0.
$$
This argument shows that $A_h(r)/r$ is strictly increasing on
$[1,\, r_0]$ unless $h\equiv 0$, which in the setting of the theorem
is ruled out by that $h''>0$ on $(1,\, \rmax)$.
\end{proof}

\subsubsection{Step 3} 
Let $\rho >0$ and $\mu(s,\rho)\in [0,\,k]$ be the total time that the
slice $u_s$ spends under the level $r=r^+ -\rho$. In other words,
\[
  \mu(s, \rho):= \Leb \big\{ t \mid r(u(s,t)) \leq r^+ -\rho \big\}.
\]
The purpose of this step is to prove the inequality
\begin{equation}
\label{eq:step3}
\mu(s, \rho) \rho \leq \frac{ r^+}{A_h(r^+)} E(u).
\end{equation}
By the maximum principle, \eqref{eq:max_prince}, we have
\begin{equation}
\label{eq:step3_1}
\int_0^k r(u_s) \,dt \leq \big(k-\mu(s,\rho) \big) r^+
+ \mu(s,\rho) (r^+ -\rho) = kr^+ - \mu(s,\rho) \rho.
\end{equation}
Next, combining \eqref{eq:step2}, Remark \ref{rmk:step2} and
\eqref{eq:step3_1}, we obtain
\[
  kr^+ - \frac{r^+}{A_h(r^+)} E(u) \leq kr^+ - \mu (s,\rho) \rho,
\]
from which \eqref{eq:step3} immediately follows. 

\subsubsection{Step 4} In this step we finish the proof of Theorem
\ref{prop:location}. Let $\eps>0$ and $C'>0$ be as in Remark
\ref{rmk:salamon}, and assume that $E(u)<\eps$. If
$ r^- \leq \inf_{\R \times S^1_k} r(u)$, there is nothing to
show. Hence we can assume without loss of generality that
$ r^- > \inf r(u) $.  Fix $\eta \in \big(0,\, r^- -\inf r(u)\big)$.

Next, we choose $(s_0, t_0) \in \R \times S_k^1$ with
$r\big(u(s_0,t_0)\big) < r^- - \eta$. Notice that by the
Bourgeois--Oancea monotonicity property, \eqref{eq:monotone}, the
slice $u_{s_0}$ has to rise at least to the level $r=r^-$. We will
need the following claim.

\smallskip

\noindent
\textit{Claim 2:} For every $t_1 \in S_k^1$ with
$r\big(u(s_0, t_1)\big) \geq r^--\eta/2$, we have
\[
\eta/2  <  C' \sqrt{r^+} E(u)^{1/4} \vert t_1 -t_0 \vert.
\]

\smallskip

First, let us prove Theorem \ref{prop:location} assuming the
claim. Set $\rho = (r^+ - r^-) + \eta/2$. From the claim 
and
Bourgeois--Oancea monotonicity, \eqref{eq:monotone}, we obtain
\begin{equation}
\label{eq:step4_1}
\eta/2\leq  C' \sqrt{r^+} E(u)^{1/4} \mu(s_0, \rho). 
\end{equation}
Note that \eqref{eq:monotone} guarantees the existence of a
  time $t_1 \in S_k^1$ as in the claim. Combining \eqref{eq:step4_1}
with \eqref{eq:step3}, we see that
\[
  \eta^2/4\leq (r^+ - r^-)\eta/2 + \eta^2/4
  \leq \frac{C'\cdot (r^+)^{3/2}}{A_h(r^+)} E(u)^{5/4}.
\]
Next, we take the square root and get
\begin{equation}
\label{eq:step4_2}
\eta \leq
\frac{ \sqrt{4C'}\cdot (r^+)^{3/4} } {\sqrt{A_h(r^+)}} E(u)^{5/8}.
\end{equation}
Since \eqref{eq:step4_2} holds for all
$\eta \in (0,\, r^- -\inf r(u))$, it holds in particular for
$\eta= r^- -\inf r(u)$. Setting $C=\sqrt{4C'}$ completes the proof of
the theorem.

\begin{proof}[Proof of Claim 2]
We have
\begin{equation}
  \label{eq:eta}
  \eta/2 < r\big(u(s_0, t_1)\big) - r\big(u(s_0, t_0)\big) =
  \int_{t_0}^{t_1} dr \big(\partial_t u(s_0,t)\big)\, dt.
\end{equation}
Let us rewrite the integral using the identity 
\[
  dr(\partial_t u) = r(u)\alpha(\partial_s u) +
  r(u)\alpha\big(JX_H(u)\big) = r(u)\alpha(\partial_s u),
\]
which follows from \eqref{eq:complex_strc}, the Floer equation
\eqref{eq:floer_2} and the condition that $\alpha(JX_H)=0$. Then
\eqref{eq:eta} turns into
\begin{align*}
  \eta/2 &< \int_{t_0}^{t_1} r\big(u(s_0, t)\big)\alpha
           \big(\partial_s u(s_0, t)\big) \, dt \\
         & \leq  \int_{t_0}^{t_1} r\big(u(s_0, t)\big)
           \frac{\|\partial_s u(s_0, t)\|}
           {\big\| R_{\alpha} \big(u(s_0,t)\big)\big\|}\, dt \\
         &= \int_{t_0}^{t_1} \sqrt{r\big(u(s_0, t)\big)}
           \| \partial_s u(s_0, t) \|\, dt \\
         & \leq   \sqrt{r^+}  \max_{t_0 \leq t\leq t_1}
           \| \partial_s u(s_0, t) \|\cdot | t_1-t_0| \\
         & \leq  C \sqrt{r^+}  E(u)^{1/4} \vert t_1-t_0 \vert,
\end{align*}
where $R_\alpha$ is the Reeb vector field of $\alpha$. In the
  second line we compare $\|\partial_s u(s_0, t)\|$ with the length of
  the projection of $\partial_s u(s_0, t)$ in the Reeb
  direction and in the third we have used the identity
\[
\| R_{\alpha}\|^2 = \omega(R_{\alpha}, JR_{\alpha})= r.
\]
The fourth line follows from the maximum principle,
\eqref{eq:max_prince}, and the fifth line is a consequence of Remark
\ref{rmk:salamon}. This concludes the proof of the claim and the
theorem.
\end{proof}

\subsubsection{Proof of the Bourgeois--Oancea monotonicity,
  \eqref{eq:monotone}}
\label{sec:BO}
Here, for the sake of completeness, we include the proof of
\eqref{eq:monotone} closely following the argument from \cite[Lemma
2.3]{CO}. Arguing by contradiction, assume that there exists
$s_0 \in \R$ such that
\[
\max_{t \in S^1_k}  r\big(u(s_0,t) \big) < r^-.
\]
It follows from the maximum principle, \eqref{eq:max_prince}, that
\[
\max_{t \in S^1_k}  r\big(u(s,t) \big) \leq r^-
\] 
for all $s \in [s_0, \infty)$. Using \eqref{eq:step1} and the
condition that $A_h'(r) \geq 0$ we see that
\[
  \frac{d}{ds} \int_0^k r (u_s)\, dt \leq -  kA_h(r^-)
  +\int_0^k A_h\big(r(u_{s})\big)\, dt \leq 0
\]
for all $s \in [s_0,\, \infty)$. On the other hand, by the assumption,
we have
\[
\int_0^k  r\big(u(s,t)\big)\, dt \leq  kr^--\eps.
\]
for $s=s_0$ and some $\eps>0$. Then the same is true for all
$s \in [s_0,\, \infty)$. This contradicts the fact that the
left-hand side converges to $k r^-$ as $s\to\infty$,
completing the proof of \eqref{eq:monotone}. \hfill \qed

\subsection{Energy bounds}
\label{sec:energy}
There are three sufficiently different approaches to the proof of
crossing energy type results. The first one is based on target Gromov
compactness; see \cite{Fi}. This approach is used in the original
proof in \cite{GG:hyperbolic} and then in \cite{CGG:Entropy, GG:PR,
  Me:Entropy} and more recently and in a more sophisticated form in
\cite{CGP}. The second approach uses the upper bound
\eqref{eq:salamon} from, e.g., \cite{Sa90, Sa}, quoted in Remark
\ref{rmk:salamon} and is also closely related on the conceptual level
to (the proof of) Gromov compactness. This method is pointed out in
\cite[Rmk.\ 6.4]{CGG:Entropy}. The third approach utilized in
\cite{GGM} is based on finite-dimensional approximations in Morse
theory. This approach does not fit well with general Floer theory
techniques, but either of the first two methods can be used to prove
Theorem \ref{thm:CE}. Here we have chosen the second one, which is
more hands-on and direct, albeit somewhat less general.

\begin{proof}[Proof of Theorem \ref{thm:CE}]
  Let $\tz=(z,r_*)$ and $H$ be as in the statement of the theorem. Let
  also $u$ be a Floer cylinder of $H^{\sharp k}$ asymptotic to $\tz^k$
  at either end. Our goal is to show that the energy $E(u)$ cannot be
  arbitrarily small. Clearly, to this end, we can require $E(u)$ to be
  sufficiently small which we will do several times throughout the
  process. Furthermore, we can assume that $k$ is large enough, since
  for a fixed $k\in \N$ the desired lower bound directly follows from
  Gromov compactness; see Section \ref{sec:CE}.

  Setting $r^-_*=r_*=r^+$ and taking a sufficiently small $\delta>0$,
  we see that by \eqref{eq:h'''1} the condition \eqref{eq:h'''2} of
  Theorem \ref{thm:location} is satisfied. As a consequence, the image
  of $u$ is contained in $M\times (r_*-\delta,\,r_*+\delta)$
  when $E(u)$ is sufficiently small. From now on we will assume that
  this is the case.

  Next, \eqref{eq:salamon} from Remark \ref{rmk:salamon} holds since
  $E(u)$ is assumed to be small. Therefore,
  $$
  \big\| \partial_s u(s, t) \big\| < C' E(u)^{1/4}=:\eps,
  $$
  for some constant $C'$. (This $\eps$ is unrelated to $\eps$ and
  $\eps'$ in Remark \ref{rmk:salamon}.) By the Floer equation,
  \eqref{eq:floer_2},
  $$
  \| \p_t u - X_H\| < \eps,
  $$
  for all $s\in\R$, where $X_H=h' R_\alpha$. (This is another point
  where it is more convenient to work with $H^{\sharp k}$ than with
  $kH$; for $X_{H^{\sharp k}}=X_H$ is independent of $k$.) Denote by
  $v$ the projection of $u$ to $M$. Then we also have
    \begin{equation}
      \label{eq:difference0}
  \big\| \p_t v - h'(r(u))R_\alpha \big\| < \eps.
\end{equation}

Fix a compact isolating neighborhood $V$ of $z$ in
$M$. Without loss of generality, we can assume that $V$ is a closed
tubular neighborhood of $z(\R)$ in $M$. We note that $v$ cannot be
entirely contained in $V$ or, equivalently, $u$ cannot be contained in
$U:=V\times(r_*-\delta,\,r_*+\delta)$. Indeed, if this were
the case, $u$ would be asymptotic to some orbit $\tz^m$ of
$H^{\sharp k}$ with $m\neq k$ at the other end. (We cannot have $m=k$ since the action difference must be non-zero.)
This, however, is impossible, for the free homotopy classes of these
two orbits are different in $U$ because the free homotopy
classes of $z^k$ and $z^m$ are different in $V$. Hence the image of
$v$ intersects $\p V$.

Furthermore, it is not hard to see that, since $V$ is isolating, there
exists $\tau_0>0$ and $\eta >0$ such that
  \begin{equation}
    \label{eq:V-flow}
  \max_{\tau\in [-\tau_0,\,\tau_0]}
  d\big(V,\varphi_\alpha^\tau(x)\big)>\eta
\end{equation}
for all $x\in \p V$. Here $d$ is some fixed distance on $M$ and
$\varphi_\alpha$ is the Reeb flow of $\alpha$.

Since $u$ is asymptotic to $\tz^k$, for some $s\in\R$ the curve
$v(s,S^1_k)$ is contained in $V$ while $v(s,t)\in \p V$ for some $t\in S^1_k$. Without loss of generality, we may assume that
$t=0$. Thus, $v(s,t)\in V$ for all $t\in S^1_k$ and $v(s,0)\in \p
V$. Let us reparametrize the map $t\mapsto v(s,t)$ by using the change
of variables $t=t(\tau)$ so that $$ t'(\tau)= h'\big(r(u(s,t))\big),
  $$
  and set
  $$
  \gamma(\tau)=v(s,t(\tau)).
  $$

  The curve $\gamma$ is parametrized by the circle $S^1_{\tau_k}$ with
$$
\frac{k}{h'(r_*+\delta)}
\leq \frac{k}{\max_{t\in S^1_k} h'\big(r(u(s,t))\big)}\leq \tau_k
\leq \frac{k}{\min_{t\in S^1_k}  h'\big(r(u(s,t))\big)}
\leq \frac{k}{h'(r_*-\delta)}.
$$
Here the lower bound follows from the maximum principle and the upper
bound from Theorem \ref{thm:location}. We will take $k$ large enough
so that $\tau_k\geq 2\tau_0$. The image of $\gamma$ is contained in
$V$ and $x:=\gamma(0)\in \p V$.

We infer from \eqref{eq:difference0} and Theorem
\ref{thm:location} that
$$
\big\| \dot{\gamma}(\tau)- R_\alpha\big( \gamma(\tau) \big) \big\|
< \frac{\eps}{\min_{t\in S^1_k} h'\big(r(u(s,t))\big)}
\leq \frac{\eps}{h'(r_*-\delta)},
$$
where the dot stands for the derivative with respect to $\tau$. This
is the point where Theorem \ref{thm:location} becomes crucial. By the
standard Gronwall inequality type argument (see, e.g.,
\cite{Br, GG:PR}) and the above upper bound,
\begin{equation}
  \label{eq:gamma-phi}
  d\big(\gamma(t), \varphi^{\tau}(x)\big)
  \leq \frac{e^{C |\tau|}\eps}{h'(r_*-\delta)},
\end{equation}
for some constant $C$. (In particular, $\gamma$ is a pseudo-orbit of
the Reeb flow, but we will not directly use this fact.) Combining
\eqref{eq:V-flow} and \eqref{eq:gamma-phi}, we obtain
$$
\eta\leq \frac{e^{C |\tau|}\eps}{h'(r_*-\delta)}
$$
for some $\tau\in [-\tau_0,\,\tau_0]$. As a consequence,
$$
\eta\leq \frac{e^{C \tau_0}\eps}{h'(r_*-\delta)},
$$
which gives a lower bound on $\eps$ and hence on $E(u)$, independent
of $k$. More explicitly, we have
$$
E(u) \geq \big[\eta h'(r_*-\delta) e^{-C \tau_0}/ C'\big]^4>0.
$$
This concludes the proof of Theorem \ref{thm:CE}.
\end{proof}

\section{Appendix: The local Le Calvez--Yoccoz
  theorem in dimension two}
\label{sec:LCY}
For the sake of completeness, here we include a topological proof of
the local version of the Le Calvez--Yoccoz theorem, \cite{LCY},
closely following \cite[Prop.\ 3.1]{Fr} and \cite{FM}.

To state the result, let us recall a few classical notions from
dynamics. The index $i(\varphi,z)$ of an isolated fixed point $z$ of a
smooth map $\varphi\colon \R^n\to\R^n$ is defined as the Brower degree
of the associated map
\begin{align*}
\partial B^n(z,\epsilon)\to \partial B^n(0,1),\qquad
y\mapsto \frac{\varphi(y)-y}{\|\varphi(y)-y\|},
\end{align*}
for any $\epsilon>0$ small enough. Here $B^n(z,\epsilon)$ is the ball
of radius $\epsilon$ centered at $z$. The definition extends to
isolated fixed points of smooth maps of manifolds by employing
arbitrary local coordinates. If $\varphi\colon M\to M$ is a
diffeomorphism, we denote by $\inv(\varphi,U)$ the largest
$\varphi$-invariant subset contained in $U\subset M$, i.e.
\begin{align*}
\inv(\varphi,U):=\bigcap_{k\in\Z} \varphi^k(U).
\end{align*}
A compact $\varphi$-invariant subset $\Lambda$ is called \emph{locally
  maximal} when $\inv(\varphi,U)=\Lambda$ for some neighborhood
$U\subset M$ of $\Lambda$; cf.\ Section \ref{sec:CE}.  A fixed point
$z$ of $\varphi$ is called \emph{irrationally elliptic} when the
eigenvalues of the linear map $d\varphi(z)$ lie on the unit circle of
the complex plane and are not roots of 1. This is equivalent to that
$\varphi$ is elliptic and all iterates $\varphi^k$ or $\varphi$ are
non-degenerate at $z$.

\begin{Theorem}[\cite{Fr,FM}]
  \label{thm:LCY-F}
  An irrationally elliptic fixed point of a surface symplectomorphism
  is never a locally maximal invariant set.
\end{Theorem}

Applying this theorem to the Poincar\'e return map of a closed
irrationally elliptic orbit in dimension three, we obtain the
following result.

\begin{Corollary}
  \label{cor:LCY-F}
  A closed irrationally elliptic orbit of a volume-preserving flow on
  a 3-manifold is never a locally maximal set.
\end{Corollary}

\begin{proof}[Proof of Theorem \ref{thm:LCY-F}]
  Let $z$ be an irrationally elliptic fixed point of a surface
  symplectomorphism $\varphi$. In particular, $d\varphi^m(z)$ has
  eigenvalues in $S^1\setminus\{1\}$ for all $m\geq 1$, and $z$ is an
  isolated fixed point of $\varphi^m$. Its indices are given by
  \begin{equation}
    \label{eq:index-varphi}
    i(\varphi^m,z)=1,\qquad\forall m\geq 1,
  \end{equation}
  see, e.g., \cite[p.\ 320]{KH}. Let us assume by contradiction that
  $z$ is a locally maximal $\varphi$-invariant subset. Under this
  assumption, a theorem due to Easton~\cite{Ea} implies that $z$
  possesses so-called isolating blocks in the sense of Conley
  theory. More precisely, following Franks and Misiurewicz \cite[Sec.\
  3]{Fr}, there exist an arbitrarily small connected compact
  neighborhood $V\subset M$ of $z$ and an arbitrarily small compact
  neighborhood $W\subset V\setminus\{z\}$ of the exit set
  \[
    V^-:=\big\{y\in V\ \big|\
    \varphi(y)\not\in\mathit{interior}(V)\big\}
  \] 
  such that $\inv(\varphi,\overline{V\setminus W})=\{z\}$,
  $\varphi(W)\cap\overline{V\setminus W}=\varnothing$, and $(V,W)$ is
  homeomorphic to a finite simplicial pair. Since $\varphi$ is
  area-preserving, we must have $V^-\neq\varnothing$, and
  thus \[W\neq\varnothing.\] If $\partial V\subset V^-$, since $\phi$
  is area-preserving, we must have $\varphi(V)=V$, but this would
  contradict the local maximality of $z$. Therefore,
  $\partial V\not\subset V^-$ and we can choose $W$ small enough so
  that
\begin{align}
\label{e:pV_not_in_W}
 \partial V\not\subset W.
\end{align}
We consider the induced quotient map
\begin{align*}
\widetilde\varphi:V/W\to V/W.
\end{align*}
Notice that this map has precisely two fixed points: $z$ and
$[W]$. Since $(V,W)$ is a finite simplicial pair, we can apply
Lefschetz fixed-point theorem to the iterates of $\widetilde\varphi$
and obtain
\begin{align}
\label{e:Lefschetz}
  \tr(\widetilde\varphi_0^m)-\tr(\widetilde\varphi_1^m)
  +\tr(\widetilde\varphi_2^m)
  = i(\varphi^m,z)+i(\widetilde\varphi^m,[W]),
\end{align}
where $\widetilde\varphi_k:\H_k(V/W;\Q)\to \H_k(V/W;\Q)$ is the
homomorphism induced by $\widetilde\varphi$. Since $V/W$ is connected,
we have
\[
  \tr(\widetilde\varphi_0)=1.
\]
By~\eqref{e:pV_not_in_W}, $\H_2(V/W;\Q)=0$. Therefore, 
\[
  \tr(\widetilde\varphi_2)=0.
\]
For any sufficiently small neighborhood $U\subset V/W$ of $[W]$, we
have $\widetilde\varphi(U)=[W]$, and hence the index of the fixed
point $[W]$ is given by
\begin{align*}
i(\widetilde\varphi^m,[W])=1.
\end{align*}
Putting these equations together and using \eqref{eq:index-varphi}, we
see that equation~\eqref{e:Lefschetz} becomes
\begin{align*}
 \tr(\widetilde\varphi_1^m)=-1.
\end{align*}
This is impossible. Indeed, a simple linear algebra argument
  \cite[Lemma~2.3]{Fr} shows that, for every linear endomorphism $L$
of a finite-dimensional vector space, $\tr(L^m)\geq0$ for infinitely
many integers $m\geq1$.
\end{proof}

\begin{Remark}
  The requirement that $z$ is irrationally elliptic in Theorem
  \ref{thm:LCY-F} is essential. It is easy to construct an
  area-preserving diffeomorphism $\varphi\colon \R^2\to\R^2$ with a
  non-degenerate rationally elliptic fixed point $z$ such that
  $\varphi^k$ has a monkey saddle at $z^k$, for some $k\in\N$, and
  hence $z$ is locally maximal.
\end{Remark}

\begin{Remark}
  It is worth pointing out that the proof of Theorem \ref{thm:LCY-F}
  is purely two-dimensional. The dimension assumption is crucially
  used in the fact that the homology $\H_*(V/W;\Q)$ is concentrated in
  only two degrees. As we have noted in the introduction it is not
  known if Theorem \ref{thm:LCY-F} holds in higher dimensions. In
  other words, there may exist a Hamiltonian diffeomorphism $\varphi$
  with an elliptic fixed point $z$ which is non-degenerate for all
  iterates $\varphi^m$ and locally maximal. This seems particularly
  likely when the linearization $D\varphi$ at $z$ is not required to
  be semi-simple.
\end{Remark}

%

\newpage

\end{document}